\documentclass[12pt,reqno]{amsart}
\usepackage{amssymb}
\usepackage{amsmath}
\usepackage{amscd,amsxtra,amsthm}
\usepackage[all]{xy}
\usepackage{etex}
\usepackage{pictex}
\usepackage{graphicx}
\usepackage{mathtools}
\usepackage{float}
\usepackage{xcolor}
\usepackage[utf8]{inputenc}
\usepackage[T1]{fontenc}
\usepackage{subcaption}
\usepackage{tikz}
\usepackage{tkz-tab}
\usepackage{caption}

\usepackage{hyperref}
\hypersetup{colorlinks, citecolor=blue}

\usepackage{mathtools}
\usepackage{multirow}
\DeclarePairedDelimiter{\ceil}{\lceil}{\rceil}
\DeclarePairedDelimiter{\floor}{\lfloor}{\rfloor}
\usepackage{comment}

\newtheorem{lemma}{Lemma}[section]

\newtheorem{theorem}[lemma]{Theorem}

\theoremstyle{conjecture}

\theoremstyle{definition}

\theoremstyle{remark}

\theoremstyle{observation}

\numberwithin{equation}{section}

\begin{document}

\title[Weakened Gallai-Ramsey Numbers for Books]{Weakened Gallai-Ramsey Numbers for Books}

\author{Mark Budden}
\address{Department and Mathematics and Computer Science, Western Carolina University, Cullowhee, NC, USA.}
\email{mrbudden@email.wcu.edu}

\author{Adam Gregory}
\address{Department and Mathematics and Computer Science, Western Carolina University, Cullowhee, NC, USA.}
\email{gregory@wcu.edu}

\subjclass[2020]{Primary: {05C55, 05D10}; Secondary: {05C15}}
\keywords{Ramsey number, Gallai coloring, book graph} 

\begin{abstract}
For $1\le s<t$ and any graph $G$, the weakened Gallai-Ramsey number $gr^t_s(G)$ is defined to be the least  $p\in \mathbb{N}$ such that every Gallai $t$-coloring of the edges of $K_p$ (i.e., a $t$-coloring that lacks rainbow triangles) contains a subgraph isomorphic to $G$ whose edges use at most $s$ of the colors.  In the case of a book graph $B_n:=K_2+nK_1$, Jakhar and Moun determined the values $gr^3_2(B_3)=6$ and $gr^3_2(B_4)=7$.  In this paper, we extend their results to $t>3$ colors, and we determine the values of $gr^3_2(B_n)$ for $5\le n\le 15$.  General lower bounds for $gr^t_2(B_n)$ are also given.
\end{abstract}

\maketitle

\section{Introduction}\label{intro}


Denote the complete graph of order $n$ by $K_n$, and let $mK_n$ denote the disjoint union of $m$ copies of $K_n$.  Given two graphs $G_1$ and $G_2$, the {\it join} $G_1+G_2$ is the graph with vertex set $V(G_1)\cup V(G_2)$ and edge set $$E(G_1)\cup E(G_2)\cup \{xy \ | \ x\in V(G_1) \ \mbox{and} \ y\in V(G_2)\}.$$  For $n\in \mathbb{N}$, define the {\it book} $B_n:=K_2+nK_1$, which has order $n+2$ and size $2n+1$.  In this definition, the $K_2$-subgraph is called the {\it spine} of the book, and the vertices in the $nK_1$-subgraph are called the {\it pages} of the book.  Note that $B_1\cong K_3$ and $B_2\cong K_4-e$, where $K_n-e$ is the graph formed by taking $K_n$ and deleting a single edge.


A {\it $t$-coloring} of $K_n$ is a map $$f:E(K_n)\longrightarrow \{1 ,2, \dots , t\},$$ which is not assumed to be surjective (i.e., fewer than $t$ colors may be used).  If $t\le 3$, we will often refer to colors $1$, $2$, and $3$, as red, blue, and green, respectively.  
For graphs $G_1, G_2, \dots , G_t$, the {\it $t$-color Ramsey number} $r(G_1, G_2, \dots , G_t)$ is defined to be the least $p\in \mathbb{N}$ such that every $t$-coloring of $K_p$ contains a monochromatic subgraph in color $i$ that is isomorphic to $G_i$, for some $i\in \{1, 2, \dots , t\}$.  When $G_1=G_2=\dots =G_t$, we shorten the notation to $r^t(G_1)$.    A $t$-coloring of $K_{r(G_1, G_2, \dots , G_t)-1}$ that avoids a monochromatic copy of $G_i$ in color $i$, for all $i\in \{1, 2, \dots , t\}$, is called a {\it critical coloring for $r(G_1, G_2, \dots , G_t)$}.  
An overview of the known $2$-color Ramsey numbers for books can be found in
Section 5.3 of Radziszowski's dynamic survey \cite{Rad}.



A {\it Gallai $t$-coloring} of $K_n$ is a $t$-coloring $f$ of $K_n$ that avoids rainbow triangles.  A {\it rainbow triangle} is a subset $\{x, y, z\}\subseteq V(K_n)$ such that no two of $f(xy)$, $f(xz)$, and $f(yz)$ are equal.  For any graph $G$, the {\it Gallai-Ramsey number} $gr^t(G)$ is defined to be the least $p\in \mathbb{N}$ such that every Gallai $t$-coloring of $K_p$ contains a monochromatic subgraph that is isomorphic to $G$.  A Gallai $t$-coloring of $K_{gr^t(G)-1}$ that avoids a monochromatic copy of $G$ is called a {\it critical coloring for $gr^t(G)$}.   Since every Gallai $t$-coloring of $K_p$ is a $t$-coloring of $K_p$, it follows that $gr^t(G)\le r^t(G)$, for every graph $G$.  

In 2019, Zou, Mao, Magnant, Wang, and Ye \cite{ZMMWY} considered Gallai-Ramsey numbers for books, proving the following theorem.
\begin{theorem}[\cite{ZMMWY}]
For $t, n\in \mathbb{N}$ such that $2\le n\le 5$,
$$gr^t(B_n)=\begin{cases} n+2 & \mbox{if $t=1$} \\ (r^2(B_n)-1)\cdot 5^{(t-2)/2}+1 & \mbox{if $t$ is even} \\ 2\cdot (r^2(B_n)-1)\cdot 5^{(t-3)/2}+1 & \mbox{if $t\ge 3$ is odd.}\end{cases}$$
\end{theorem}
\noindent The authors of \cite{ZMMWY} also provided general upper and lower bounds for $gr^t(B_n)$, when $n\ge 6$.  Their upper bounds were later improved by Li and Li \cite{LL} in 2024.

In this paper, we consider a weakened version of $gr^t(B_n)$.  That is, rather than finding the threshold for which a monochromatic book is guaranteed, we seek the threshold for which a book is guaranteed whose edges use at most two colors.  More generally, let $1\le s<t$ and define the {\it weakened Gallai-Ramsey number} $gr^t_s(G)$ to be the least $p\in \mathbb{N}$ such that every Gallai $t$-coloring of $K_p$ contains a subgraph that is isomorphic to $G$ whose edges use at most $s$ colors.  Note that $gr^t(G)=gr^t_1(G)$.  For $s_1\le s_2$, the inequality $$gr^t_{s_2}(G)\le gr^t_{s_1}(G)$$ implies that the existence of weakened Gallai-Ramsey numbers follows from the existence of Gallai-Ramsey numbers.   At present, weakened Gallai-Ramsey numbers have been studied for complete graphs (\cite{BB}, \cite{BW}, \cite{JM}, and \cite{LBW}), for cycles (\cite{BW} and \cite{JM}), and for wheels, books, and complete bipartite graphs \cite{JM}.

In the case of books, Jakhar and Moun \cite{JM} proved that
$$gr^3_2(B_3)=6 \qquad \mbox{and} \qquad  gr^3_2(B_4)=7.$$
Our work herein builds on that of~\cite{JM}.  We begin by proving several structural lemmas for Gallai colorings, and general lower bounds for $gr^t_2(B_n)$ in Section \ref{genboundsec}.  In Section \ref{B3andB4}, we generalize Jakhar and Moun's \cite{JM} results to any $t \geq 3$.  In Section \ref{main}, we determine $gr_2^3(B_n)$, when $5 \leq n \leq 15$.  Section \ref{conclusion} concludes with a discussion about several directions that our work may be extended.

\section{Structural Results and General Bounds}\label{genboundsec}

One benefit of working with Gallai $t$-colorings of $K_n$ (as opposed to $t$-colorings of $K_n$) is that the proofs of upper bounds for Gallai-Ramsey numbers can be broken into manageable cases by using the following structural theorem (which is a reinterpretation of a result due to Gallai \cite{Gallai}).  

\begin{theorem}[\cite{GS}]\label{Gallaistruct}
Every Gallai-colored complete graph can be formed by replacing the vertices in a $2$-colored complete graph of order at least $2$ with Gallai-colored complete graphs.
\end{theorem}

In the statement of Theorem \ref{Gallaistruct}, the $2$-colored complete graph with order at least $2$ is called the {\it base graph}, and the Gallai-colored complete graphs that replace the vertices in the base graph are called the {\it blocks}.
The partitioning of the vertex set of the Gallai-colored complete graph into the vertex sets of the blocks is called a {\it Gallai partition}.    All edges joining any distinct pair of blocks are necessarily the same color.   The following four lemmas will be useful in simplifying the work needed to prove upper bounds for $gr^t_2(B_n)$ when using Theorem \ref{Gallaistruct}.

\begin{lemma}[Lemma 3.1, \cite{MN}]\label{not3}
If $\mathcal{B}$ is the base graph of a Gallai partition, chosen to have minimal order, then $|V(\mathcal{B})|\ne 3$.
\end{lemma}

\begin{lemma}\label{RemoveLargeBaseGraph}
Let $m,t \in \mathbb{N}$ such that $t\ge 3$ and consider a Gallai $t$-coloring of $K_m$ with base graph $\mathcal{B}$ satisfying $|V(\mathcal{B})|=k\ge 4$.  If $n\in \mathbb{N}$ such that $3\le n<m$ satisfies $k>\frac{2m}{m-n+1}$, then there exists a $B_n$-subgraph whose edges use at most two colors.
\end{lemma}

\begin{proof}
Let $X_1, X_2, \dots , X_k$ be the vertex sets for the blocks in this Gallai $t$-coloring, and suppose that no pair of blocks contain a total of at most $m-n$ vertices.  Then every pair of blocks must contain a total of at least $m-n+1$ vertices (i.e., $|X_i\cup X_j|\ge m-n+1$ for all $i\ne j$).  Denote by $\mathfrak{X}^2$ the set of all pairs $\{X_i, X_j\}$, where $i\ne j$.  If we sum over all such pairs, we find that $$\mathop{\sum}\limits_{\{X_i,X_j\}\in \mathfrak{X}^2} |X_i\cup X_j|\ge {k\choose 2}(m-n+1).$$  In this sum, each block occurs in exactly $k-1$ of the pairs, from which it follows that $$\mathop{\sum}\limits_{\{X_i,X_j\}\in \mathfrak{X}^2} |X_i\cup X_j|=(k-1)m,$$ resulting in the inequality $$(k-1)m\ge  {k\choose 2}(m-n+1).$$  This inequality is equivalent to $$\frac{2m}{m-n+1}\ge k,$$ which contradicts the inequality assumed in the statement of the lemma.  It follows that there exists a pair of blocks such that $|
X_i\cup X_j|\le m-n$ for some $i\ne j$.  A $B_n$-subgraph whose edges use at most two colors (the colors in the base graph) can then be formed with a spine joining a vertex in $X_i$ to a vertex in $X_j$, and with pages given by $n$ of the vertices in $\mathfrak{X}^2\setminus \{X_i, X_j\}$ (since $|\mathfrak{X}^2\setminus \{X_i, X_j\}|\ge m-(m-n)=n$).
\end{proof}

\begin{lemma}\label{lem:no-four}
Let $m,n\in \mathbb{N}$ such that $n\ge 3$ and $m\ge \frac{5n}{3}$.  Then every Gallai $3$-coloring of $K_m$, whose base graph has order $4$, contains a $B_n$-subgraph whose edges use at most two colors. 	
\end{lemma}

\begin{proof}  For $n\ge 3$, consider a Gallai $3$-coloring of $K_m$, with $m\ge \frac{5n}{3}$, in which the base graph $\mathcal{B}$ has order $4$.  Without loss of generality, suppose that the edges in $\mathcal{B}$ use the colors red and blue.  Denote the vertex sets for the blocks by $X_1$, $X_2$, $X_3$, and $X_4$, with $k_i:=|X_i|$ and indexing such that $k_1 \geq k_2\geq k_3 \geq k_4$.  We claim that either $$k_1 + k_2 \geq n \quad \mbox{or} \quad k_2 + k_3 + k_4 \geq n+2.$$ To see that this is true, suppose that $$k_1 + k_2 \le n-1 \quad \mbox{and} \quad k_2 + k_3 + k_4 \le n+1.$$   Adding these inequalities together yields
$$\frac{5n}{3}+ k_2 \le (k_1 + k_2) + (k_2 + k_3 + k_4) \leq (n-1) + (n+1) = 2n,$$
which implies that
\begin{equation}
\label{eqn:b2UB}
k_2 \leq \frac{n}{3}.
\end{equation}

On the other hand, since $k_2 \geq k_3 \geq k_4$, we have $k_3 + k_4 \leq 2 k_2$.  Moreover, since $k_1 + k_2 \le n-1$, we have that $k_1 \leq n-1-k_2$.  Hence,
\begin{equation}
\label{eqn:mUB}
\frac{5n}{3} \le k_1 + k_2 + k_3 + k_4 \leq (n-1-k_2) + k_2 + 2k_2 = n-1+2k_2.
\end{equation}
Combining Inequalities \eqref{eqn:b2UB} and \eqref{eqn:mUB} implies that
$$\frac{5n}{3} \leq n-1 + 2k_2 \leq n - 1 + 2 \left(\frac{n}{3}\right) = \frac{5n}{3} - 1,$$ which is 
a contradiction.  Thus, one of $$k_1 + k_2 \geq n \quad \mbox{or} \quad k_2 + k_3 + k_4 \geq n+2$$ must hold. 

\underline{Case 1:} Assume that $k_1 + k_2 \geq n$.  Since $k_3 \geq k_4 \geq 1$, taking any edge joining $X_3$ and $X_4$ to be the spine yields a red/blue book with at least $n$ pages (coming from $X_1\cup X_2$). 

\underline{Case 2:} Assume that $k_2 + k_3 + k_4 \geq n+2$.  Suppose that for some $i\in \{1 ,2 ,3, 4\}$, the subgraph induced by $X_i$ contains a red or blue edge.  Using such an edge as the spine, we obtain a red/blue book with at least $n$ pages (coming from $\bigcup_{j\ne i} X_j$).  Here, we are using the observation that $$\mathop{\sum}\limits_{j\ne i} k_j \ge k_2+k_3+k_4\ge n+2.$$  If no such edge exists, then every block is a green complete subgraph.  The $K_3$-subgraph in $\mathcal{B}$ corresponding with the blocks $X_2$, $X_3$, and $X_4$ has two edges that receive the same color.  Without loss of generality, suppose that the edges joining $X_4$ to $X_2\cup X_3$ are all red.  Then selecting any edge $xy$ in the subgraph induced by $X_4$ as the spine, a red/green book can be formed with pages $X_2\cup X_3\cup (X_4\setminus \{x,y\})$.  Such a book contains at least $n$ pages.
\end{proof}

\begin{lemma}\label{mainlem2}
Let $m,n\in \mathbb{N}$ such that $n\ge 6$ and $m\ge \frac{5n-1}{3}$.  Then every Gallai $3$-coloring of $K_m$ in which there exists a vertex $x$ that is only incident with edges in one color, and whose vertex set with $x$ deleted has a base graph of minimal order $4$, contains a $B_n$-subgraph whose edges use at most two colors. 	
\end{lemma}

\begin{proof}  Assume that $n\ge 6$.  It suffices to prove the lemma for the case where $m=\ceil{\frac{5n-1}{3}}$, since for any $N\in \mathbb{N}$ satisfying $N\ge m$, we know that $K_N$ contains $K_m$ as a subgraph.  Consider a Gallai $3$-coloring of $K_m$, in which there exists a vertex $x$ that is only incident with red edges, and whose vertex set with $x$ deleted has a base graph of minimal order $4$ (i.e., it is not possible to describe the Gallai $3$-coloring of the $K_{m-1}$ with a base graph of order less than $4$).  Denote the base graph of this $K_{m-1}$ by $\mathcal{B}$ and denote the vertex sets for its blocks by $X_1$, $X_2$, $X_3$, and $X_4$, with $k_i:=|X_i|$ and indexing such that $k_1 \geq k_2\geq k_3 \geq k_4$.  We consider two cases, based on whether or not the edges in $\mathcal{B}$ include the color red.

\underline{Case 1:} Suppose that the edges in $\mathcal{B}$ are red and another color (say, blue).  Since $X_4$ is the block with least order, it follows that $$k_4\le \floor[\bigg]{\frac{m-1}{4}}\le \floor[\bigg]{\frac{5n-1}{12}}.$$  Hence, \begin{align} k_1+k_2+k_3&\ge \frac{5n-1}{3}-1-\floor[\bigg]{\frac{5n-1}{12}} \notag \\
&\ge \frac{20n-4}{12}-\left( \frac{5n-1}{12}+1\right) \notag \\ &=\frac{15n-15}{12},\notag\end{align}
and this last expression is greater than or equal to $n$ if and only if $n\ge 5$, which is assumed to be true.  In this case, a red/blue $B_n$ can be formed with spine $xy$, where $y\in X_4$, and $n$ pages in $X_1\cup X_2\cup X_3$.

\underline{Case 2:} Suppose that the edges in $\mathcal{B}$ are blue and green.
We claim that $k_1 + k_2 \geq n$ or $k_1 + k_i + k_j \geq n+1, \mbox{for all distinct } i,j\in \{2, 3, 4\}$. 
To see that this claim is true, suppose that $k_1 + k_2 \le n-1$ and $k_1 + k_i + k_j \le n$, for some distinct $i,j\in \{ 2, 3, 4\}$.   Since $k_1\ge k_2\ge k_3\ge k_4$, the first inequality implies that $k_1+k_i\le n-1$, for all $i\in \{2, 3, 4\}$.
Adding these inequalities together yields
$$\frac{5n-1}{3}-1+ k_1 \le (k_1 + k_k) + (k_1 + k_i + k_j) \leq (n-1) + n = 2n-1,$$
for some distinct $i,j,k\in \{2, 3, 4\}$.  This is equivalent to 
\begin{equation}
k_1 \le 2n-\frac{5n-1}{3}=\frac{n+1}{3}, \notag
\end{equation}
which along with
$$k_2+k_3+k_4\le k_1+k_3+k_4\le n,$$ leads to the inequality 
\begin{equation}
\frac{5n-1}{3}-1\le k_1+(k_2+k_3+k_4)\le \frac{n+1}{3}+n=\frac{4n+1}{3}.
\notag
\end{equation}
This is equivalent to $n\le 5$, contradicting the assumption that $n\ge 6$.  Therefore,  $k_1 + k_2 \geq n$ or $k_1 + k_i + k_j \geq n+1$, for all distinct $i,j\in \{2, 3, 4\}$.

\underline{Subcase 2.1:} Assume that $k_1 + k_2 \geq n$.  Since $k_3 \geq k_4 \geq 1$, taking any edge joining $X_3$ and $X_4$ to be the spine yields a blue/green book with at least $n$ pages (coming from $X_1\cup X_2$). 

\underline{Subcase 2.2:} Assume that $k_1 + k_i + k_j \geq n+1, \mbox{for all distinct } i,j\in \{2, 3, 4\}$.  If the subgraph induced by $X_2$ has a blue or green edge, then such an edge forms the spine of a blue/green book with at least $n+1$ pages given by $X_1\cup X_3\cup X_4$.  So, assume that all of the edges in the subgraph induced by $X_2$ are red.  Applying a similar argument to $X_3$ and $X_4$, we assume that the subgraphs induced by $X_3$ and by $X_4$ only contain red edges.

If the edges joining $X_2$ to $X_1$ are the same color as those joining $X_2$ to either $X_3$ or $X_4$ (say, $X_3$), then  a book can be formed whose edges use at most two colors (one of which is red) by using the edge $xy$, with $y\in X_2$ as the spine, and with pages $X_1\cup X_3 \cup (X_2 \setminus \{y\})$.  Here, we are using the assumption that $k_1+k_2+k_3\ge n+1$ to see that such a book has at least $n$ pages.  So, without loss of generality, suppose that the edges joining $X_2$ to $X_1$ are blue and the edges joining $X_2$ to $X_3 \cup X_4$ are green.  Using a similar argument for the set $X_3$, we can assume that the edges joining $X_3$ to $X_1$ are blue and edges joining $X_3$ to $X_4$ are green (see Figure \ref{Lemconstruct1}).
\begin{figure}[h!]
\centerline{
\includegraphics[width=0.5\textwidth]{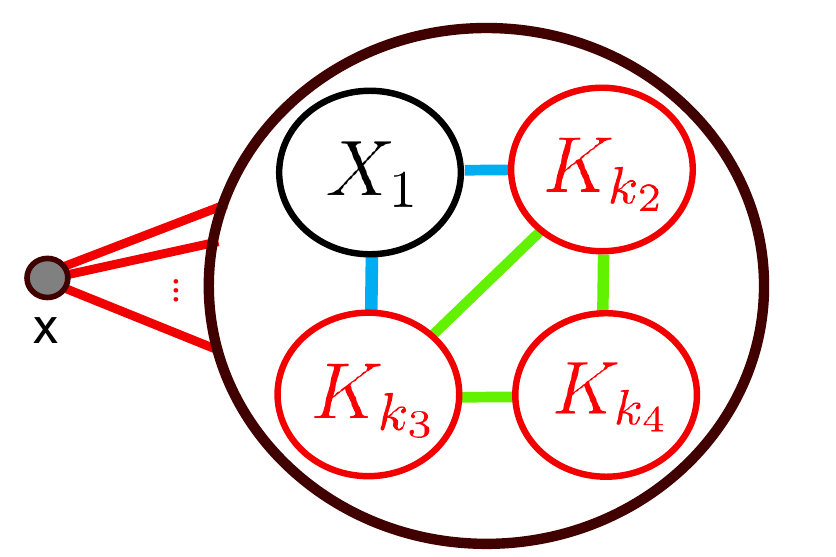}}
\caption{The Gallai $3$-coloring of $K_{m}$ corresponding to the construction in Subcase 2.2 of the proof of Lemma \ref{mainlem2}.}\label{Lemconstruct1}
\end{figure}
If the edges joining $X_1$ to $X_4$ are blue, then the $K_{m-1}$ can be formed with a base graph of order $2$ having blocks $X_1$ and $X_2\cup X_3\cup X_4$, contradicting the minimality assumption on $\mathcal{B}$.  Likewise, if the edges joining $X_1$ to $X_4$ are green, then the $K_{m-1}$ can be formed with a base graph of order $2$ having blocks $X_1\cup X_2\cup X_3$ and $X_4$, contradicting the minimality assumption on $\mathcal{B}$.  

It follows that all of the Gallai $3$-colorings considered in this lemma contain a $B_n$-subgraph whose edges use at most two colors.
\end{proof}

Now we turn our attention to proving two theorems that provide general lower bounds for $gr^t_2(B_n)$.

\begin{theorem}\label{thm:General-Lower-Bound}
For $t,n\in \mathbb{N}$ such that $t\ge 3$ and $n\ge 3$, $$gr^t_2(B_n)\ge \begin{cases} \frac{4n+6}{3} & \mbox{if $n\equiv 0\pmod{3}$} \\ \frac{5n+1}{3} & \mbox{if $n\equiv 1\pmod{3}$} \\ \frac{4n+7}{3} & \mbox{if $n\equiv 2\pmod{3}$.}\end{cases}$$
\end{theorem}

\begin{proof} Since every Gallai $3$-coloring of a complete graph is a Gallai $t$-coloring, it is sufficient to prove the inequality when $t=3$.  Consider the following cases.

\underline{Case 1:} Suppose that $n\equiv 0\pmod{3}$.  Let $m=\frac{4n+6}{3}$ and note that $4$ divides $m-2=\frac{4n}{3}$.  Consider the Gallai $3$-coloring of $K_{m-2}$ formed by first taking the $2$-colored $K_4$ in colors red and blue shown in Figure \ref{gen1.1}, and then replacing each of its vertices with green $K_{\frac{m-2}{4}}$-subgraphs.  Next, introduce one additional vertex and join it to the existing $K_{m-2}$ with green edges (see Figure \ref{gen1.1}).  
\begin{figure}[h!]
\centering
\begin{subfigure}{.27\textwidth}
    \centering
    \includegraphics[width=1\linewidth]{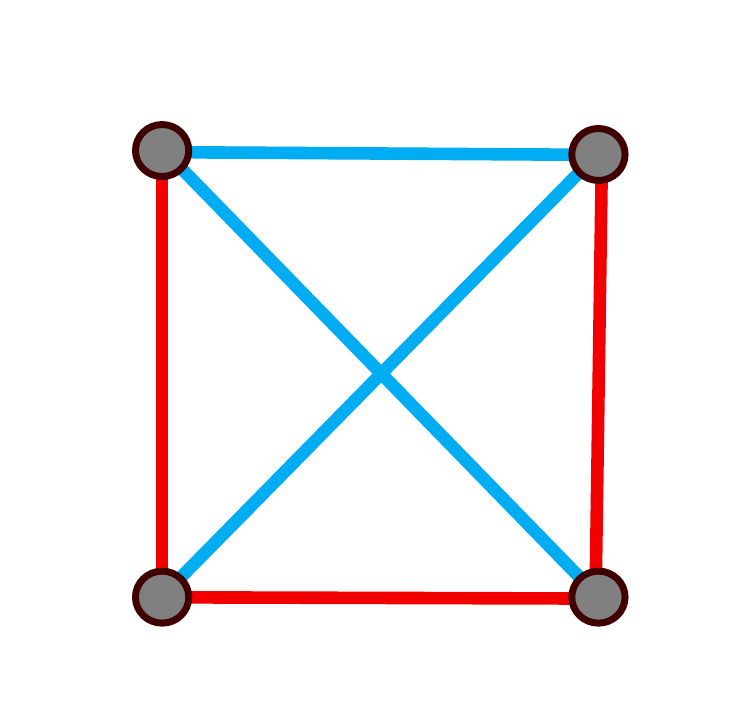}  
    \caption{}
    \label{gen1.1}
\end{subfigure}\qquad \qquad
\begin{subfigure}{.5\textwidth}
    \centering
    \includegraphics[width=1\linewidth]{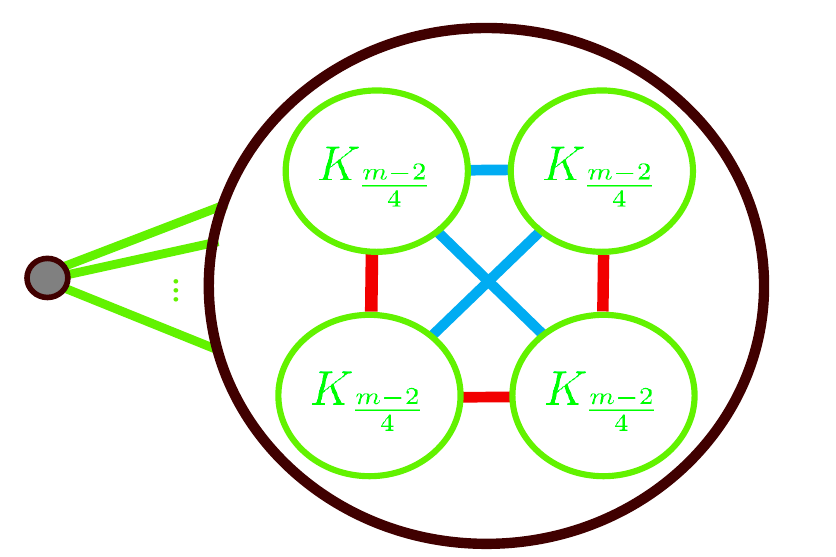}  
    \caption{}
    \label{gen1.2}
\end{subfigure}
\caption{Constructing a Gallai $3$-coloring of $K_{m-1}$ that avoids a $B_n$-subgraph whose edges use at most two colors, where $n\equiv 0\pmod{3}$.}\label{Case1lower}
\end{figure}
%
In the resulting $K_{m-1}$, the largest red/blue book has $\frac{m-2}{2}=\frac{4n}{6}<n$ pages, the largest red/green book has $\frac{3m-10}{4}=n-1$ pages, and the largest blue/green book has $\frac{3m-10}{4}=n-1$ pages.  It follows that $gr^t_2(B_n)\ge \frac{4n+6}{3}$.

\underline{Case 2:} Suppose that $n\equiv 1\pmod{3}$.  Let $m=\frac{5n+1}{3}$ and note that $5$ divides $m-2=\frac{5n-5}{3}$.  Consider the Gallai $3$-coloring of $K_{m-2}$ formed by first taking the $2$-colored $K_5$ in colors red and blue shown in Figure \ref{gen2.1}, and replacing each of its vertices with green $K_{\frac{m-2}{5}}$-subgraphs.  
\begin{figure}[h!]
\centering
\begin{subfigure}{.27\textwidth}
    \centering
    \includegraphics[width=1\linewidth]{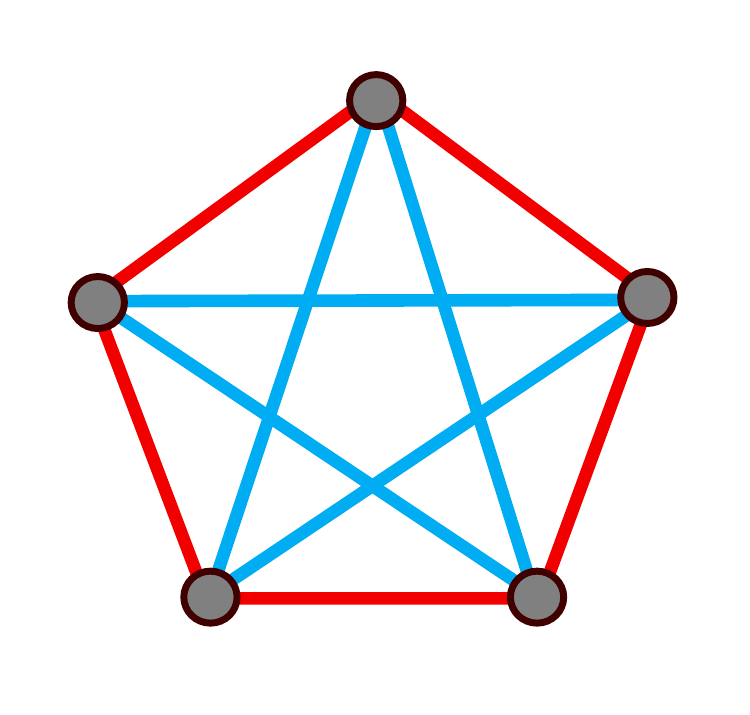}  
    \caption{}
    \label{gen2.1}
\end{subfigure}\qquad \qquad
\begin{subfigure}{.5\textwidth}
    \centering
    \includegraphics[width=1\linewidth]{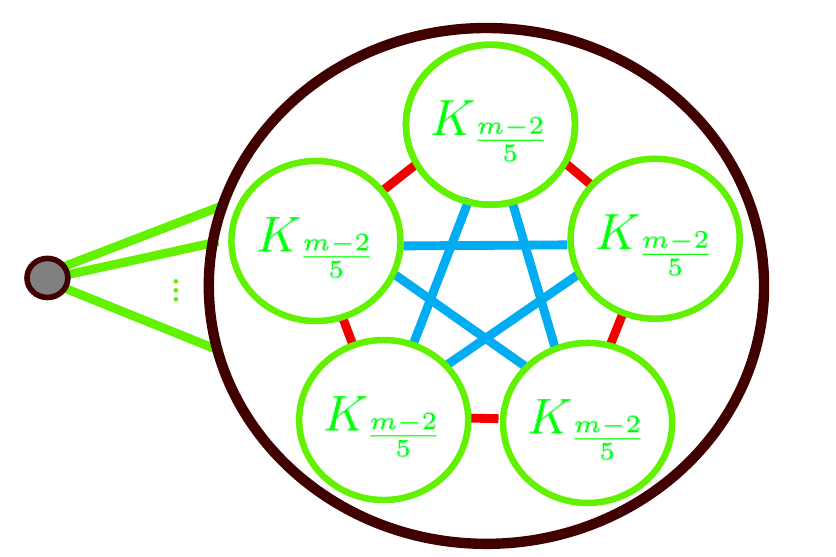}  
    \caption{}
    \label{gen2.2}
\end{subfigure}
\caption{Constructing a Gallai $3$-coloring of $K_{m-1}$ that avoids a $B_n$-subgraph whose edges use at most two colors, where $n\equiv 1\pmod{3}$.}\label{Case2lower}
\end{figure}  Next, introduce one additional vertex and join it to the existing $K_{m-2}$ with green edges (see Figure \ref{gen2.2}).  In the resulting $K_{m-1}$, the largest red/blue book has $\frac{3(m-2)}{5}=n-1$ pages, the largest red/green book has $\frac{3m-11}{5}=n-2$ pages, and the largest blue/green book has $\frac{3m-11}{5}=n-2$ pages.  It follows that $gr^t_2(B_n)\ge \frac{5n+1}{3}$.

\underline{Case 3:} Suppose that $n\equiv 2\pmod{3}$.  Let $m=\frac{4n+7}{3}$ and note that $4$ divides $m-1=\frac{4n+4}{3}$. Consider the Gallai $3$-coloring of $K_{m-1}$ formed by taking the $2$-colored $K_4$ in colors red and blue shown in Figure \ref{gen1.1}, and replacing each of its vertices with green $K_{\frac{m-1}{4}}$-subgraphs (see Figure \ref{Case3lower}).  
\begin{figure}[h!]
\centerline{
\includegraphics[width=0.6\textwidth]{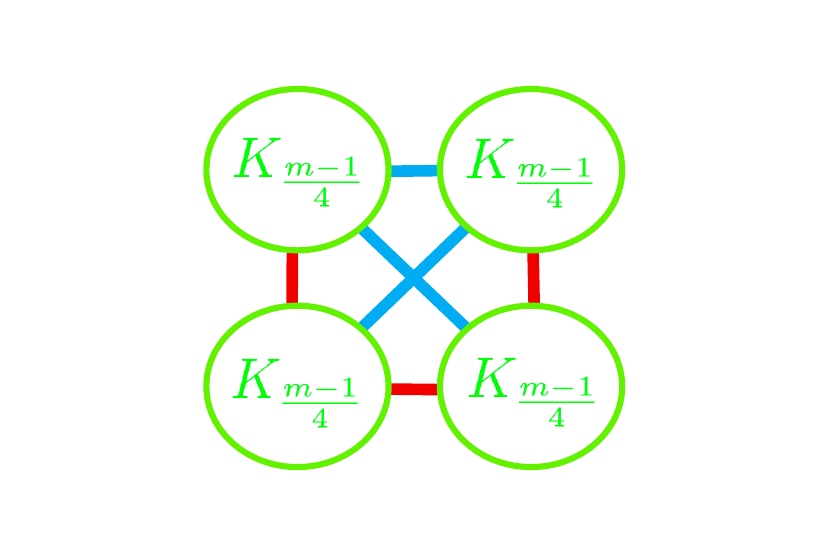}}
\caption{A Gallai $3$-coloring of $K_{m-1}$ that avoids a $B_n$-subgraph whose edges use at most two colors, where $n\equiv 2\pmod{3}$.}\label{Case3lower}
\end{figure}
In the resulting $K_{m-1}$, the largest red/blue book has $\frac{m-1}{2}=\frac{2(n+1)}{3}<n$ pages, the largest red/green book has $\frac{3m-11}{4}=n-1$ pages, and the largest blue/green book has $\frac{3m-11}{4}=n-1$ pages.  It follows that $gr^t_2(B_n)\ge \frac{4n+7}{3}$.
\end{proof}

Using known values of $2$-color Ramsey numbers for books, we obtain the following theorem. 

\begin{theorem}\label{critlower}
Let $n,t,k\in \mathbb{N}$ such that $n\ge 2$, $t\ge 3$, and $k< r(B_n, B_n)$.  If there exists a $2$-coloring of $K_{k}$ that avoids a monochromatic $B_n$ and a monochromatic $K_{1,\floor{\frac{n+k-1}{2}}+1}$, then $gr^t_2(B_{n+k})\ge 2k+1$.
\end{theorem}

\begin{proof} Let $\mathcal{G}$ be a $2$-coloring of $K_k$ in blue and green that avoids a monochromatic $B_n$ and a monochromatic $K_{1,\floor{\frac{n+k-1}{2}}+1}$.   Next, take a red $K_2$ and replace each of its vertices with copies of $\mathcal{G}$.  In the resulting $K_{2k}$, the largest red/blue book with a red spine has $2\floor[\big]{\frac{n+k-1}{2}}\le n+k-1$ pages.  The largest red/blue book with a blue spine has $n-1+k$ pages.  Thus, no red/blue $B_{n+k}$ exists.  By a similar argument, no red/green $B_{n+k}$ exists.  The largest blue/green book has $k-2<n+k$ pages.  It follows that $gr^3_2(B_{n+k})> 2k$.  Since every Gallai $3$-coloring of a complete graph is a Gallai $t$-coloring, the stated inequality holds for all $t\ge 3$.
\end{proof}

\section{The Books $B_3$ and $B_4$}\label{B3andB4}

In this section, we will determine the values of $gr^t_2(B_3)$ and $gr^t_2(B_4)$, for all $t\ge 3$.  This extends the work of Jakhar and Moun \cite{JM}, where these numbers were determined when $t=3$.

\begin{theorem}\label{B3-t}
	For all $t\in \mathbb{N}$ such that $t \geq 3$, we have $gr_2^t(B_3) = 6$.
\end{theorem}

\begin{proof}
The lower bound $gr_2^t(B_3) \geq 6$ follows from Theorem~\ref{thm:General-Lower-Bound}.  
To prove the upper bound, consider a Gallai $t$-coloring of $K_6$ that avoids a $B_3$-subgraph whose edges use at most two colors.  Using Theorem \ref{Gallaistruct}, we let $\mathcal{B}$ be the base graph, chosen to have minimal order.   By Lemmas \ref{not3} and \ref{RemoveLargeBaseGraph}, it suffices to consider the case where $|V(\mathcal{B})|=2$.

By the pigeonhole principle, some block contains at least three vertices.  If both blocks have order at least $2$, then it is possible to select two vertices from one block to form the spine of a $B_3$-subgraph, and three vertices from the other block to form the pages.  Such a $B_3$-subgraph necessarily uses at most two colors on its edges.

So, assume that one block contains a single vertex $x$, and denote the vertex set of the larger block by $Y=\{y_1, y_2, y_3, y_4, y_5\}$.  Without loss of generality, suppose that the edges joining $x$ to $Y$ are all in color $1$ (red).  The subgraph induced by $Y$ is a Gallai $t$-coloring of $K_5$, and we assume that it avoids a $B_3$-subgraph whose edges use at most two colors.  Using Theorem \ref{Gallaistruct}, we let $\mathcal{B}'$ be its base graph, chosen to have minimal order.  By Lemmas \ref{not3} and \ref{RemoveLargeBaseGraph}, it suffices to consider the case where $|V(\mathcal{B'})|=2$.  If one block in $\mathcal{B}'$ has order two and the other has order three, then the vertices from the smaller block form the spine, and the vertices from the larger block form the pages, of a $B_3$-subgraph whose edges use at most two colors.  So, one block in $\mathcal{B}'$ has order $1$ and, without loss of generality, we assume it contains the vertex $y_1$.  Regardless of the color of the edges that join $y_1$ to $Y\setminus \{y_1\}$, the $B_3$-subgraph with spine $xy_1$ and pages $y_2$, $y_3$, and $y_4$ uses at most two colors.    It follows that $gr^t_2(B_3)\le 6$.
\end{proof}

\begin{theorem}\label{grB4}
For all $t\in \mathbb{N}$ such that $t\ge 3$, $$gr^t_2(B_4)=\begin{cases} 7 & \mbox{if $t=3$} \\ 8 & \mbox{if $t\ge 4$.}\end{cases}$$
\end{theorem}

\begin{proof}
The $t=3$ case was proved by Jakhar and Moun \cite{JM} (see also Theorem 3.29 of \cite{B2}).  So, assume that $t\ge 4$ and consider the Gallai $4$-coloring of $K_7$ given in Figure \ref{LowerB4}.  
\begin{figure}[h!]
\centerline{
\includegraphics[width=0.5\textwidth]{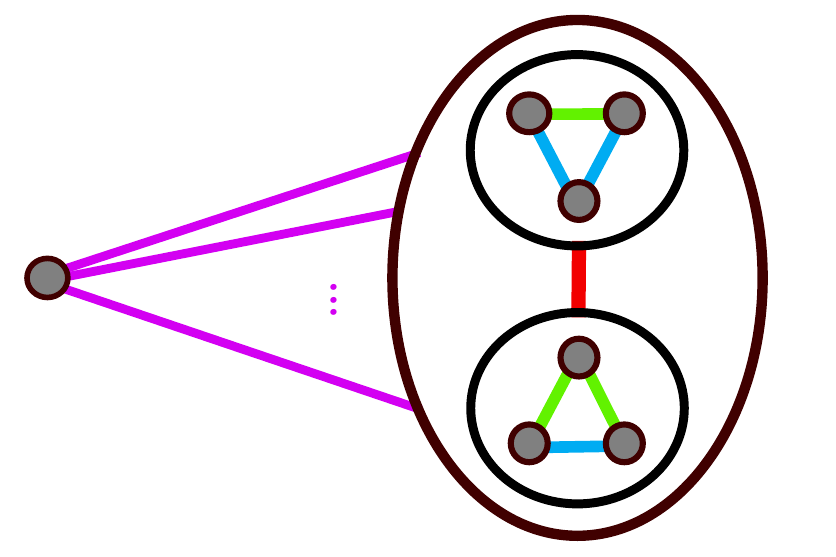}}
\caption{A Gallai $4$-coloring of $K_{7}$ that avoids a $B_4$ whose edges use at most two colors.}\label{LowerB4}
\end{figure}
Since it lacks a $B_4$-subgraph whose edges use at most two colors, it follows that $gr^t_2(B_4)\ge 8$, when $t\ge 4$.  
To prove the upper bound, consider a Gallai $t$-coloring of $K_8$ that avoids a $B_4$-subgraph whose edges use at most two colors.  Using Theorem \ref{Gallaistruct}, we let $\mathcal{B}$ be the base graph, chosen to have minimal order.   By Lemmas \ref{not3} and \ref{RemoveLargeBaseGraph}, it suffices to consider the case where $|V(\mathcal{B})|=2$.

By the pigeonhole principle, some block contains at least four vertices.  If both blocks have order at least $2$, then it is possible to select two vertices from one block to form the spine of a $B_4$-subgraph, and four vertices from the other block to form the pages.  Such a $B_4$-subgraph necessarily uses at most two colors on its edges.

So, assume that one block contains a single vertex $x$, and denote the vertex set of the larger block by $Y=\{y_1, y_2, y_3, y_4, y_5, y_6, y_7\}$.  Without loss of generality, suppose that the edges joining $x$ to $Y$ are all in color $1$ (red).  The subgraph induced by $Y$ is a Gallai $t$-coloring of $K_7$, and we assume that it avoids a $B_4$-subgraph whose edges use at most two colors.  Using Theorem \ref{Gallaistruct}, we let $\mathcal{B}'$ be its base graph, chosen to have minimal order.  By Lemmas \ref{not3} and \ref{RemoveLargeBaseGraph}, it suffices to consider the case where $|V(\mathcal{B'})|=2$.  

By the pigeonhole principle, some block in $\mathcal{B}'$ contains at least four vertices.  If both blocks in $\mathcal{B}'$ have order at least $2$, then it is possible to select two vertices from one block to form the spine of a $B_4$-subgraph, and four vertices from the other block to form the pages.  Such a $B_4$-subgraph necessarily uses at most two colors on its edges.
If one block contains only a single vertex, say $y_1$, then regardless of the color of the edges joining $y_1$ to $Y\setminus \{y_1\}$, the $B_4$-subgraph with spine $xy_1$ and pages $y_2$, $y_3$, $y_4$, and $y_5$ uses at most two colors.  Therefore, every Gallai $t$-coloring of $K_8$ contains a $B_4$-subgraph whose edges use at most two colors, and it follows that $gr^t_2(B_4)\le 8$, when $t\ge 4$.
\end{proof}

\section{The Books $B_n$, for $n\ge 5$}\label{main}

As we saw in Theorem \ref{grB4}, the value of $gr^t_2(B_4)$ is dependent upon the value of $t$.  For this reason, as we begin to consider larger books, we restrict our attention to the case $t=3$.  As we shall soon see, the lemmas and theorems in Section \ref{genboundsec} do most of the heavy lifting for us.  

\begin{theorem}
$gr^3_2(B_5)=9$.
\end{theorem}

\begin{proof}
The lower bound $gr^t_2(B_5)\ge 9$ follows from Theorem \ref{thm:General-Lower-Bound}. To prove the reverse inequality, consider a Gallai $3$-coloring of $K_9$, with base graph $\mathcal{B}$, chosen to have minimal order. Since 
\[
\frac{2 \cdot 9}{9 - 5 + 1} < 4,
\]
then by Theorem \ref{Gallaistruct} and Lemmas \ref{not3} and \ref{RemoveLargeBaseGraph}, it suffices to consider when $|V(\mathcal{B})|=2$. Denote the vertex sets for the blocks by $X_1$ and $X_2$, with $k_i := |X_i|$ and the indexing such that $k_1 \geq k_2$. Without loss of generality, suppose that the edges joining $X_1$ and $X_2$ are red.  If $k_1 \geq 5$ and $k_2 \geq 2$, then we get a $B_5$ whose edges use at most two colors by choosing a spine in $X_2$ and any five vertices in $X_1$ as pages. Thus, it suffices to consider the case when $(k_1,k_2) = (8,1)$. Denote the single vertex in $X_2$ by $x$. Using Theorem~\ref{Gallaistruct}, let $\mathcal{B}'$ be the base graph for the subgraph induced by $X_1$, chosen to have minimal order. Since 
\[
\frac{2 \cdot 10}{10 - 6 + 1} < 4,
\]
then by Lemmas~\ref{not3} and \ref{RemoveLargeBaseGraph}, it suffices to consider $|V(\mathcal{B}')| = 2$. 
Without loss of generality, assume that the edge in $\mathcal{B}'$ is either red or blue.  Denote the vertex sets for the blocks in $\mathcal{B}'$ by $Y_1$ and $Y_2$, with $\ell_i := |Y_i|$ and the indexing such that $\ell_1 \geq \ell_2$. If $\ell_1 \geq 5$ and $\ell_2 \geq 2$, then a $B_5$ whose edges use at most two colors can be formed using an edge in $Y_2$ as a spine with any five vertices in $Y_1$ as pages. Thus, it suffices to consider the following cases. 

\underline{Case 1:} Assume that $(\ell_1,\ell_2) = (4,4)$. If there is a red or blue edge in $Y_2$, then a red/blue $B_{5}$ can be formed by using that edge as the spine and adding $x$ and the four vertices in $Y_1$ as pages. Hence, the subgraph induced by $Y_2$ is a green $K_4$, which contains a green $B_2$. Adding the four vertices in $Y_1$ as pages gives a $B_{6}$ (which contains a $B_5$) whose edges use at most two colors, one of which is green. 

\underline{Case 2:} Assume that $(\ell_1,\ell_2) = (7,1)$. Denote the single vertex in $Y_2$ by $y$. Then using the red edge $xy$ as a spine, we get a red/blue $B_7$ (which contains a $B_5$) by adding the vertices in $Y_1$ as pages.

In all cases, a $B_5$ whose edges use at most two colors exists, from which it follows that $gr_2^3(B_5) \le 9$.
\end{proof}

In order to determine $gr^3_2(B_6)$, we will need to use the (unique) critical coloring for $r^2(B_1)=6$ (e.g., see \cite{GY} and \cite{GG}).  The cycle shown in Figure \ref{GY1.1} is the Paley graph of order $5$, which is necessarily self-complementary.   
\begin{figure}[h!]
\centering
\begin{subfigure}{.37\textwidth}
    \centering
    \includegraphics[width=1\linewidth]{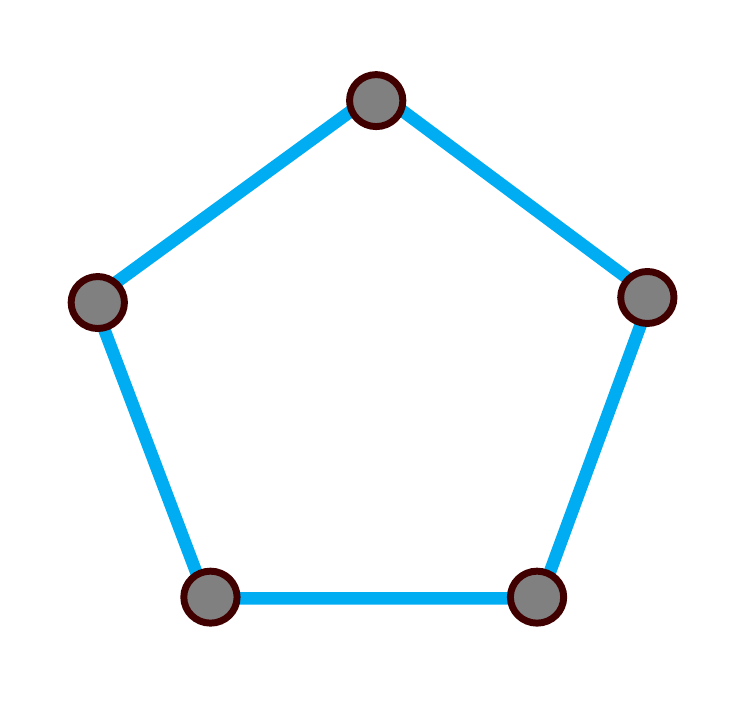}  
    \caption{}
    \label{GY1.1}
\end{subfigure}\qquad \qquad
\begin{subfigure}{.37\textwidth}
    \centering
    \includegraphics[width=1\linewidth]{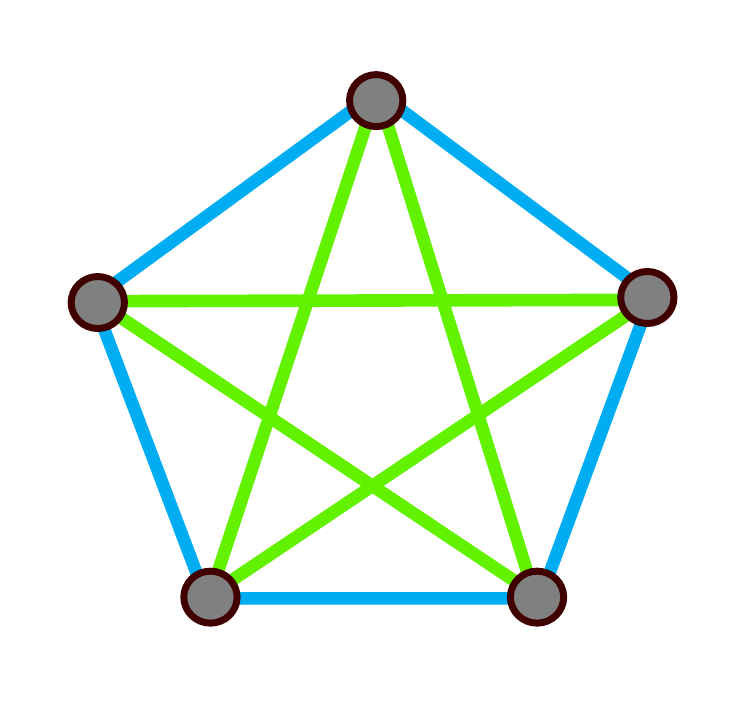}  
    \caption{}
    \label{GY1.2}
\end{subfigure}
\caption{The unique critical coloring for $r^2(B_1)=6$ (see \cite{GY} and \cite{GG}).}\label{GYlowerbound}
\end{figure}
Figure \ref{GY1.2} shows this graph and its complement in blue and green, which is the unique critical coloring for $r^2(B_1)$.


\begin{theorem}\label{thm:B6}
$gr_2^3(B_6)=11$.
\end{theorem}
\begin{proof}
The lower bound $gr_2^3(B_6) \geq 11$ follows from Theorem \ref{critlower} by letting $\mathcal{G}$ be the critical coloring for $r^2(B_1)$ in blue and green shown in Figure \ref{GY1.2}.
To prove the reverse inequality, consider a Gallai $3$-coloring of $K_{11}$, with base graph $\mathcal{B}$, chosen to have minimal order. Since
\[
\frac{2 \cdot 11}{11 - 6 + 1} < 4,
\]
then by Theorem \ref{Gallaistruct} and Lemmas \ref{not3} and  \ref{RemoveLargeBaseGraph} it suffices to consider the case where $|V(\mathcal{B})|=2$. Denote the vertex sets for the blocks by $X_1$ and $X_2$, with $k_i := |X_i|$ and the indexing such that $k_1 \geq k_2$. Without loss of generality, suppose the edges joining $X_1$ and $X_2$ are red. If $k_2 \geq 2$ and $k_1 \geq 6$, then we get a $B_6$ whose edges use at most two colors by choosing a spine in $X_2$ and any six vertices in $X_1$ as pages. Thus, it suffices to consider the case when $(k_1,k_2) = (10,1)$. Denote the single vertex in $X_2$ by $x$. Using Theorem~\ref{Gallaistruct}, let $\mathcal{B}'$ be the base graph for the subgraph induced by $X_1$, chosen to have minimal order. Since 
\[
\frac{2 \cdot 10}{10 - 6 + 1} < 5,
\]
then by Lemmas~\ref{not3} and \ref{RemoveLargeBaseGraph}, it suffices to consider the cases where $|V(\mathcal{B}')| \in \{2,4\}$. 

Let us first assume that $|V(\mathcal{B}')|=4$. If $\mathcal{B}'$ uses the two colors other than red, then since $11 \geq \frac{5 \cdot 6 - 1}{3}$ and $x$ is a vertex incident with edges in only red, Lemma~\ref{mainlem2} implies the existence of a $B_{6}$ whose edges use at most two colors. Hence, we may assume that one of the colors used in $\mathcal{B}'$ is red. Denote the vertex sets for the blocks in $\mathcal{B}'$ by $Y_1$, $Y_2$, $Y_3$, and  $Y_4$, with $\ell_i := |Y_i|$ and the indexing such that $\ell_1 \geq \ell_2 \geq \ell_3 \geq \ell_4$. Since $\ell_4 \leq \lfloor \frac{10}{4} \rfloor = 2$, it follows that $\ell_1 + \ell_2 + \ell_3 \geq 10 - 2 = 8$. Let $y$ be any vertex in $Y_4$. Taking the red edge $xy$ as a spine and adding $Y_1 \cup Y_2 \cup Y_3$ as pages, we get a book with at least eight pages whose edges use at most two colors, one of which is red.

Now suppose that $|V(\mathcal{B}')|=2$. 
Without loss of generality, assume that the edge in $\mathcal{B}'$ is red or blue.  Denote the vertex sets for the blocks in $\mathcal{B}'$ by $Y_1$ and $Y_2$, with $\ell_i := |Y_i|$ and the indexing such that $\ell_1 \geq \ell_2$. If $\ell_1 \geq 6$ and $\ell_2 \geq 2$, then a $B_6$ whose edges use at most two colors can be formed by using an edge in $Y_2$ as a spine with any six vertices in $Y_1$ as pages.  Thus, it suffices to consider the following cases. 

\underline{Case 1:} Assume that $(\ell_1,\ell_2) = (5,5)$. If there is a red or blue edge in $Y_2$, then a red/blue $B_{6}$ can be formed by adding $x$ and the five vertices in $Y_1$ as pages. Hence, the subgraph induced by $Y_2$ is a green $K_5$, which contains a green $B_3$. Adding the five vertices in $Y_1$ as pages gives a $B_{8}$ (which contains a $B_6$) whose edges use at most two colors, one of which is green. 

\underline{Case 2:} Assume that $(\ell_1,\ell_2) = (9,1)$. Denote the single vertex in $Y_2$ by $y$. Then using the red edge $xy$ as a spine, we get a red/blue $B_{9}$ (which contains a $B_6$) by adding the vertices in $Y_1$ as pages.

In all cases, a $B_6$ whose edges use at most two colors exists, from which it follows that $gr_2^3(B_6) \le 11$.
\end{proof}

\begin{theorem}
	$gr_2^3(B_7) = 12$.
\end{theorem}

\begin{proof}
The lower bound $gr_2^3(B_7) \geq 12$ follows from Theorem~\ref{thm:General-Lower-Bound}. To prove the reverse inequality, consider a Gallai 3-coloring of $K_{12}$, with base graph $\mathcal{B}$, chosen to have minimal order. Since
\[
\frac{2 \cdot 12}{12 - 7 + 1} < 5 \quad \text{and} \quad \frac{5 \cdot 7}{3} \leq 12,
\]	
then by Theorem~\ref{Gallaistruct} and Lemmas~\ref{not3}, \ref{RemoveLargeBaseGraph}, and \ref{lem:no-four}, it suffices to consider the case where $|V(\mathcal{B})|=2$. Denote the vertex sets for the blocks by $X_1$ and $X_2$, with $k_i := |X_i|$ and the indexing such that $k_1 \geq k_2$. Without loss of generality, suppose the edges joining $X_1$ and $X_2$ are red. If $k_2 \geq 7$ and $k_2 \geq 2$, then we get a $B_7$ whose edges use at most two colors by choosing a spine in $X_2$ and any seven vertices in $X_1$ as pages. Thus, it suffices to consider the following cases. 

\underline{Case 1:} Assume that $(k_1,k_2) = (6,6)$. Since $k_2 = r^2(B_1)=6$ \cite{GG}, the subgraph induced by $X_2$ must contain either a red/blue $B_1$, or a green $B_1$. Adding the six vertices in $X_1$ as pages, in the former case we get a red/blue $B_{7}$, whereas in the latter case we get a red/green $B_{7}$. 

\underline{Case 2:} Assume that $(k_1,k_2) = (11,1)$. Denote the single vertex in $X_2$ by $x$. Using Theorem~\ref{Gallaistruct}, let $\mathcal{B}'$ be the base graph for the subgraph induced by $X_1$, chosen to have minimal order. Since 
\[
\frac{2 \cdot 11}{11 - 7 + 1} < 5,
\]
then by Lemmas~\ref{not3} and \ref{RemoveLargeBaseGraph} it suffices to consider the cases where $|V(\mathcal{B}')| \in \{2,4\}$. 

Let us first assume that $|V(\mathcal{B}')|=4$. If $\mathcal{B}'$ uses the two colors other than red, then since $12 \geq \frac{5 \cdot 7 - 1}{3}$ and $x$ is a vertex incident with edges in only red, Lemma~\ref{mainlem2} implies the existence of a $B_{7}$ whose edges use at most two colors. Hence, we may assume that one of the colors used in $\mathcal{B}'$ is red. Denote the vertex sets for the blocks in $\mathcal{B}'$ by $Y_1$, $Y_2$, $Y_3$, and $Y_4$, with $\ell_i := |Y_i|$ and the indexing such that $\ell_1 \geq \ell_2 \geq \ell_3 \geq \ell_4$. Since $\ell_4 \leq \lfloor \frac{11}{4} \rfloor = 2$, it follows that $\ell_1 + \ell_2 + \ell_3 \geq 11 - 2 = 9$. Let $y$ be any vertex in $Y_4$. Taking the red edge $xy$ as a spine and adding $Y_1 \cup Y_2 \cup Y_3$ as pages, we get a book with at least nine pages, and whose edges use at most two colors, one of which is red.

Now suppose that $|V(\mathcal{B}')|=2$. 
Without loss of generality, assume that the edge in $\mathcal{B}'$ is red or blue.  Denote the vertex sets for the blocks in $\mathcal{B}'$ by $Y_1$ and $Y_2$, with $\ell_i := |Y_i|$ and the indexing such that $\ell_1 \geq \ell_2$.  If $\ell_1 \geq 7$ and $\ell_2 \geq 2$, then a $B_7$ whose edges use at most two colors can be formed using an edge in $Y_2$ as a spine with any seven vertices in $Y_1$ as pages. Thus, it suffices to consider the following subcases. 

\underline{Subcase 2.1:} Assume that $(\ell_1,\ell_2) = (6,5)$.  If $Y_2$ contains a red or blue edge, then that edges forms the spine of a red/blue $B_7$ with pages $\{x\}\cup Y_1$.  So, assume that the subgraph induced by $Y_2$ is a green $K_5$, which contains $B_3$ as a subgraph.  Adding in the vertices in $Y_1$ as pages results in a $B_9$ (which contains a $B_7$) whose edges use at most two colors, one of which is green.

\underline{Subcase 2.2:} Assume that $(\ell_1,\ell_2) = (10,1)$. Denote the single vertex in $Y_2$ by $y$. Then using the red edge $xy$ as a spine, we get a red/blue $B_{10}$ (which contains a $B_7$) by adding the vertices in $Y_1$ as pages.

In all cases, a $B_7$ whose edges use at most two colors exists, from which it follows that $gr_2^3(B_7) \le 12$.
\end{proof}

\begin{theorem}
	$gr_2^3(B_8) = 13$.
\end{theorem}

\begin{proof}
The lower bound $gr_2^3(B_8) \geq 13$ follows from the construction shown in Figure \ref{LowerB8}. 
\begin{figure}[h!]
\centerline{
\includegraphics[width=0.5\textwidth]{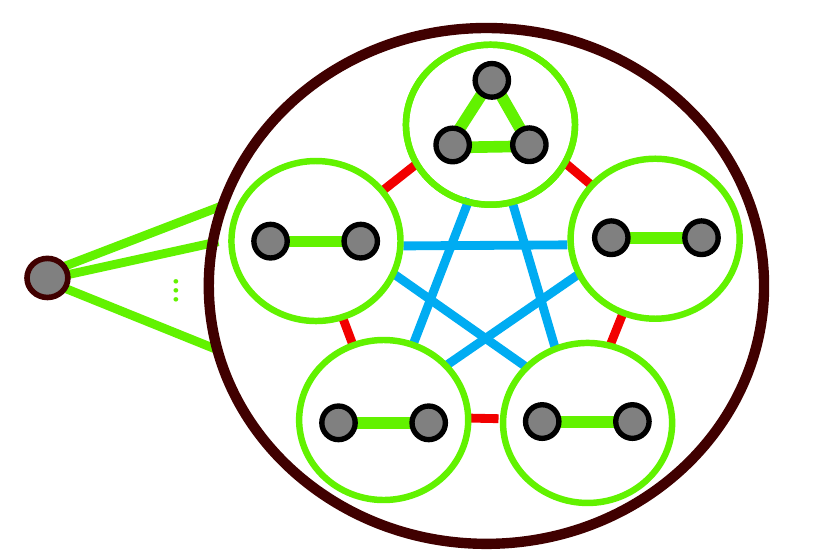}}
\caption{A Gallai $3$-coloring of $K_{12}$ that avoids a $B_8$ whose edges use at most two colors.}\label{LowerB8}
\end{figure}
Here, the largest red/blue book has seven pages, the largest red/green book has six pages, and the largest blue/green book has six pages. To prove the reverse inequality, consider a Gallai $3$-coloring of $K_{13}$, with base graph $\mathcal{B}$, chosen to have minimal order. Since 
\[
\frac{2 \cdot 13}{13 - 8 + 1} < 5,
\]
then by Theorem~\ref{Gallaistruct} and Lemmas~\ref{not3} and \ref{RemoveLargeBaseGraph}, we know that $|V(\mathcal{B})| \in \{2,4\}$. We consider the following cases. 

\underline{Case 1:} Assume that $|V(\mathcal{B})| = 4$. Without loss of generality, suppose that the $2$-coloring of $\mathcal{B}$ uses red and blue. Denote the vertex sets for the blocks by $X_1$, $X_2$, $X_3$, and $X_4$ with $k_i := |X_i|$ and the indexing such that $k_1 \geq k_2\geq k_3 \geq k_4$. If $k_1 + k_2 \geq 8$, then a $B_8$ whose edges use at most two colors is formed by taking an edge between $X_3$ and $X_4$ as a spine and choosing any $8$ vertices in $X_1 \cup X_2$ as pages. Therefore, we may assume that $k_1 + k_2 \leq 7$. If $k_1 = 5$, this forces $k_2 \leq 2$, but then $13 = k_1 + k_2+k_3+k_4 \leq 5 + 3 \cdot 2 = 11 < 13$ would be a contradiction. Likewise, if $k_1 = 6$, this forces $k_2 \leq 1$, but then $13 = k_1 + k_2+k_3 + k_4 \leq 6 + 3 \cdot 1 = 9 < 13$ is a contradiction. Moreover, since $k_2 \geq 1$, we cannot have $k_1 > 6$ without forcing $k_1 + k_2 \geq 8$. On the other hand, we cannot have $k_1 < 4$ since then $13 = k_1 + k_2+k_3 + k_4 \leq 4 \cdot 3 = 12 < 13$ would be a contradiction. It now follows that we need only consider the case where $k_1 = 4$ and $k_i = 3$ for $i \in \{2,3,4\}$. 

Suppose that $X_1$ contains a red or blue edge. Then, using it as the spine, we get a red/blue $B_{k_2+k_3+k_4} = B_9$ whose edges use at most two colors by adding the vertices in $X_2 \cup X_3 \cup X_4$ as pages. Thus, we may assume that $X_1$ contains only green edges. Since $k_1 + k_3 + k_4 \geq k_2 + k_3 + k_4$, the same argument leads us to supposing that $X_2$ contains only green edges, and likewise for $X_3$.  Among the blocks $X_1$, $X_2$, and $X_3$, there must be one that joins to the other two with edges in only one color, say red.  Assume for now that it is $X_1$. Then, as $X_1$ is a green $K_4$, which contains a green $B_2$, we can form a red/green $B_{8}$ by adding the vertices in $X_2 \cup X_3$ as pages. If instead it is $X_2$ that is incident with edges in only one color (say, red) to $X_1\cup X_3$,, then using the green $B_1$ in $X_2$ we can form a red/green $B_{8}$ by adding the vertices in $X_1 \cup X_3$ as pages. The same argument works if $X_3$ is the block incident with only one color. So in all cases, when $|V(\mathcal{B})| = 4$, there exists a $B_8$ whose edges use at most two colors. 

\underline{Case 2:} Assume that $|V(\mathcal{B})| = 2$. Denote the vertex sets for the blocks by $X_1$ and $X_2$, with $k_i := |X_i|$ and the indexing such that $k_1 \geq k_2$. Without loss of generality, suppose the edge joining $X_1$ and $X_2$ is red. If $k_2 \geq 2$ and $k_1 \geq 8$, then we get a $B_8$ whose edges use at most two colors by choosing a spine in $X_2$ and any eight vertices in $X_1$ as pages. Thus, it suffices to consider the following subcases. 

\underline{Subcase 2.1:} Assume that $(k_1,k_2) = (7,6)$. Since $k_2= r^2(B_1)$ \cite{GG}, the subgraph induced by $X_2$ must contain either a red/blue $B_1$, or a green $B_1$. Adding the seven vertices in $X_1$ as pages, in the former case we get a red/blue $B_{7}$, whereas in the latter case we get a red/green $B_{7}$. 

\underline{Subcase 2.2:} Assume that $(k_1,k_2) = (12,1)$. Denote the single vertex in $X_2$ by $x$. Using Theorem~\ref{Gallaistruct}, let $\mathcal{B}'$ be the base graph for the subgraph induced by $X_1$, chosen to have minimal order. Since 
\[
\frac{2 \cdot 12}{12 - 8 + 1} < 5,
\]
then by Lemmas~\ref{not3} and \ref{RemoveLargeBaseGraph}, it suffices to consider the cases where $|V(\mathcal{B}')| \in \{2,4\}$. 

Let us first assume that $|V(\mathcal{B}')|=4$. Denote the vertex sets of the blocks of $\mathcal{B}'$ by $Y_1$, $Y_2$, $Y_3$, and $Y_4$, with $\ell_i := |Y_i|$ and the indexing such that $\ell_1 \geq \ell_2 \geq \ell_3 \geq \ell_4$. Note that $\ell_4 \leq \frac{12}{4} = 3$ and thus $\ell_1 + \ell_2 + \ell_3 \geq 9$. If red is one of the colors used in $\mathcal{B}'$, then for any $y \in Y_4$, taking the red edge $xy$ as a spine and using $Y_1 \cup Y_2 \cup Y_3$ as pages gives a book with at least $9$ pages, and whose edges use at most two colors (one of which is red). Thus, we may assume the two colors used in $\mathcal{B}'$ are not red. Since $13 \geq \frac{5 \cdot 8 - 1}{3}$ and $x$ is a vertex incident with edges in only red, Lemma~\ref{mainlem2} implies the existence of a $B_{8}$ whose edges use at most two colors.

Now suppose that $|V(\mathcal{B}')|=2$. 
Without loss of generality, assume that the edge in $\mathcal{B}'$ is red or blue.   Denote the vertex sets for the blocks in $\mathcal{B}'$ by $Y_1$ and $Y_2$, with $\ell_i := |Y_i|$ and the indexing such that $\ell_1 \geq \ell_2$. If $\ell_1 \geq 8$ and $\ell_2 \geq 2$, then a $B_8$ whose edges use at most two colors can be formed using an edge in $Y_2$ as a spine with any eight vertices in $Y_1$ as pages. Thus, it suffices to consider the following subsubcases. 

\underline{Subsubcase 2.2.1:} Assume that $(\ell_1, \ell_2)=(6,6)$.  If there exists a vertex in $Y_1$ that is incident with at least two edges that are red or blue (say, $y_1y_2$ and $y_1y_3$ are red or blue), then a red/blue $B_8$ can be formed with spine $y_1x$ and pages in $\{y_2,y_3\}\cup Y_2$.  So, assume that each vertex in $Y_1$ is incident with at least four green edges.  Without loss of generality, assume that $y_1y_2$, $y_1y_3$, $y_1y_4$, and $y_1y_5$ are green edges in the subgraph induced by $Y_1$.  Following a similar argument, assume that $z_1z_2$, $z_1z_3$, $z_1z_4$, and $z_1z_5$ are green edges in the subgraph induced by $Y_2$.  Then a $B_8$ can be formed with spine $y_1z_1$ and pages in $\{y_2,y_3,y_4,y_5,z_2,z_3,z_4,z_5\}$ whose edges use at most two colors, one of which is green.



\underline{Subsubcase 2.2.2:} Assume that $(\ell_1,\ell_2) = (7,5)$.  If $Y_2$ contains a red or blue edge, then using that edge as the spine, a red/blue $B_8$ can be formed with pages in $\{x\}\cup Y_1$.  So, assume that the subgraph induced by$Y_2$ is a green $K_5$, which contains a $B_3$ as a subgraph.  Adding the vertices in $Y_1$ as pages to such a subgraph results in a $B_8$ whose edges use at most two colors, one of which is green.  


\underline{Subsubcase 2.2.3:} Assume that $(\ell_1, \ell_2)=(11,1)$.  Let $y$ be the vertex in $Y_2$.  Then a red/blue $B_{11}$ can be formed with spine $xy$ and pages $Y_1$.

In all cases, a $B_8$ whose edges use at most two colors exists, from which it follows that $gr_2^3(B_8)\le 13$.
\end{proof}


In order to determine $gr^3_2(B_9)$, we will need to use the critical coloring for $r^2(B_2)=10$ described in \cite{CH2}.  Start by considering the graph $K_3+K_3$, which is a Paley graph of order $9$, and also a $3\times 3$ rooks graph (see Figure \ref{CH1.1}).  Note that $K_3+K_3$ is $4$-regular.  Since it is a Paley graph, it is self-complementary, and Figure \ref{CH1.2} shows this graph and its complement in blue and green, which is a critical coloring for $r^2(B_2)$.
%
\begin{figure}[h!]
\centering
\begin{subfigure}{.37\textwidth}
    \centering
    \includegraphics[width=1\linewidth]{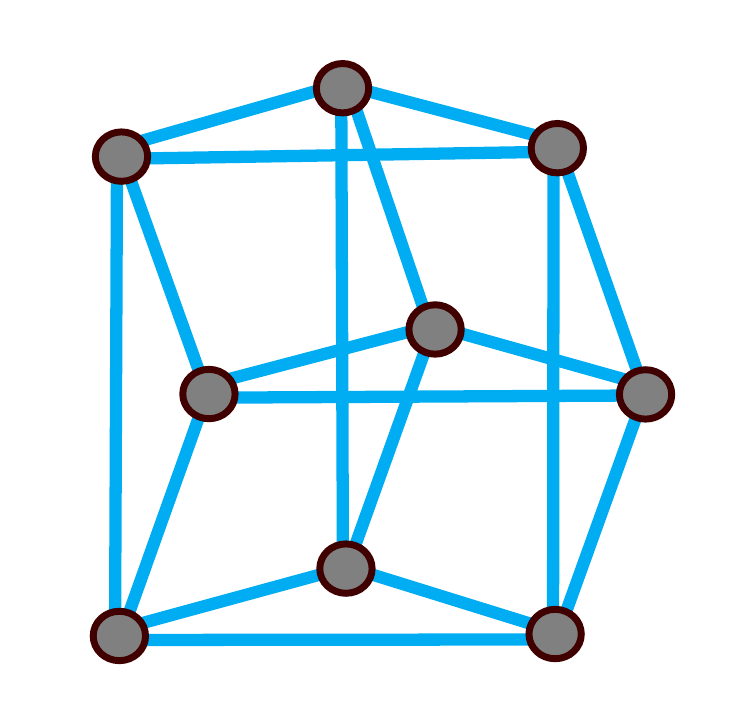}  
    \caption{}
    \label{CH1.1}
\end{subfigure}\qquad \qquad
\begin{subfigure}{.42\textwidth}
    \centering
    \includegraphics[width=1\linewidth]{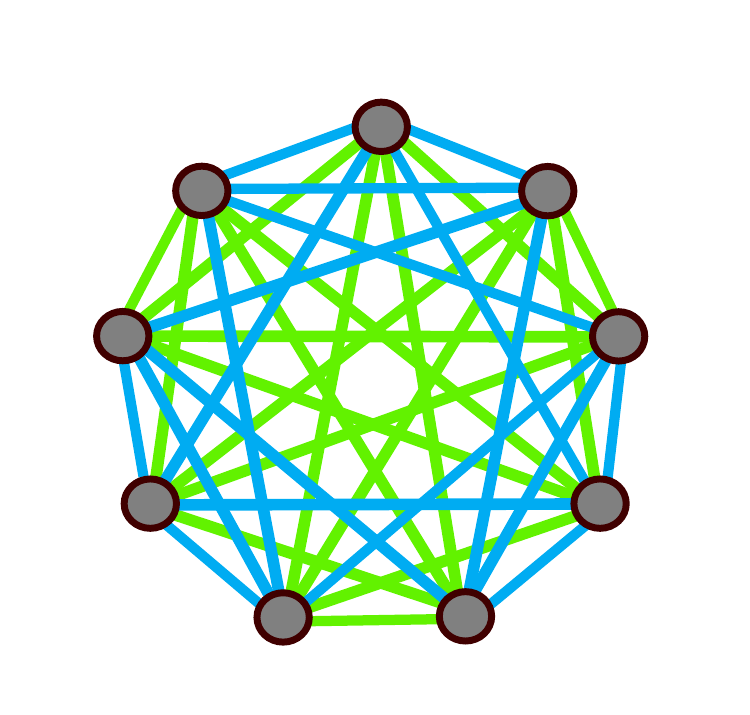}  
    \caption{}
    \label{CH1.2}
\end{subfigure}
\caption{The critical coloring for $r^2(B_2)=10$ given by Chv\'atal and Harary in \cite{CH2}.}\label{CH2lowerbound}
\end{figure}

\begin{theorem}
	$gr_2^3(B_9) = 15$.
\end{theorem}
\begin{proof}
The lower bound $gr_2^3(B_9) \geq 15$ follows from Theorem \ref{critlower} by taking $\mathcal{G}$ to be the blue/green coloring of $K_7$ formed by removing two of the vertices from the graph in Figure \ref{CH1.2}.  To prove the reverse inequality, consider a Gallai $3$-coloring of $K_{15}$, with base graph $\mathcal{B}$, chosen to have minimal order. Since
\[
\frac{2 \cdot 15}{15 - 9 + 1} < 5 \quad \text{and} \quad \frac{5 \cdot 9}{3} \leq 15,
\]	
then by Theorem~\ref{Gallaistruct} and Lemmas~\ref{not3}, \ref{RemoveLargeBaseGraph}, and \ref{lem:no-four}, it suffices to consider the case where $|V(\mathcal{B})|=2$. Denote the vertex sets for the blocks by $X_1$ and $X_2$, with $k_i := |X_i|$ and the indexing such that $k_1 \geq k_2$. Without loss of generality, suppose the edges joining $X_1$ and $X_2$ are red. If $k_1 \geq 9$ and $k_2 \geq 2$, then we get a $B_9$ whose edges use at most two colors by choosing a spine in $X_2$ and any nine vertices in $X_1$ as pages. Thus, it suffices to consider the following cases. 

\underline{Case 1:} Assume that $(k_1,k_2) = (8,7)$. Since $k_2= r(B_1,B_2)$ \cite{CH3}, the subgraph induced by $X_2$ must contain either a red/blue $B_1$, or a green $B_2$. Adding the eight vertices in $X_1$ as pages, in the former case we get a red/blue $B_{9}$, whereas in the latter case we get a red/green $B_{10}$ (which contains a $B_9$). 

\underline{Case 2:} Assume that $(k_1,k_2) = (14,1)$. Denote the single vertex in $X_2$ by $x$. Using Theorem~\ref{Gallaistruct}, let $\mathcal{B}'$ be the base graph for the subgraph induced by $X_1$, chosen to have minimal order.  Since 
\[
\frac{2 \cdot 14}{14 - 9 + 1} < 5,
\]
then by Lemmas~\ref{not3} and \ref{RemoveLargeBaseGraph}, it suffices to consider the cases where $|V(\mathcal{B}')| \in \{2,4\}$. 

Let us first assume that $|V(\mathcal{B}')|=4$. If $\mathcal{B}'$ uses the two colors other than red, then since $15 \geq \frac{5 \cdot 9 - 1}{3}$ and $x$ is a vertex incident with edges in only red, Lemma~\ref{mainlem2} implies the existence of a $B_{9}$ whose edges use at most two colors. Hence, we may assume that one of the colors used in $\mathcal{B}'$ is red. Denote the vertex sets for the blocks in $\mathcal{B}'$ by $Y_1, Y_2, Y_3, Y_4$, with $\ell_i := |Y_i|$ and the indexing such that $\ell_1 \geq \ell_2 \geq \ell_3 \geq \ell_4$. Since $\ell_4 \leq \lfloor \frac{14}{4} \rfloor = 3$, it follows that $\ell_1 + \ell_2 + \ell_3 \geq 14 - 3 = 11$. Let $y$ be any vertex in $Y_4$. Taking the red edge $xy$ as a spine and adding $Y_1 \cup Y_2 \cup Y_3$ as pages, we get a book with at least $11$ pages, and whose edges use at most two colors.

Now suppose that $|V(\mathcal{B}')|=2$. 
Without loss of generality, assume that the edge in $\mathcal{B}'$ is red or blue.  Denote the vertex sets for the blocks in $\mathcal{B}'$ by $Y_1$ and $Y_2$, with $\ell_i := |Y_i|$ and the indexing such that $\ell_1 \geq \ell_2$. If $\ell_1 \geq 9$ and $\ell_2 \geq 2$, then a $B_9$ whose edges use at most two colors can be formed using an edge in $Y_2$ as a spine with any nine vertices in $Y_1$ as pages. Thus, it suffices to consider the following subcases. 

\underline{Subcase 2.1:} Assume that $(\ell_1,\ell_2) = (7,7)$. Since $\ell_2 = r(B_1,B_2)$ \cite{CH3}, there exists a red/blue $B_1$ or a green $B_2$ in $Y_2$. In the latter case, adding the seven vertices in $Y_1$ as pages gives a blue/green $B_{9}$. In the former case, adding the seven vertices in $Y_1$, together with the vertex $x$, gives a red/blue $B_{9}$. 

\underline{Subcase 2.2:} Assume that $(\ell_1,\ell_2) = (8,6)$. Since $\ell_2 = r^2(B_1)$ \cite{GG}, there exists a red/blue $B_1$ or a green $B_1$ in $Y_2$. Adding the eight vertices in $Y_1$ as pages, in the former case we get a red/blue $B_{9}$, whereas in the latter case we get a blue/green $B_{9}$. 

\underline{Subcase 2.3:} Assume that $(\ell_1,\ell_2) = (13,1)$. Denote the single vertex in $Y_2$ by $y$. Then using the red edge $xy$ as a spine, we get a red/blue $B_{13}$ (which contains a $B_9$) by adding the vertices in $Y_1$ as pages.

In all cases, a $B_9$ whose edges use at most two colors exists, from which it follows that $gr_2^3(B_9) \le 15$.
\end{proof}

\begin{theorem}
	$gr_2^3(B_{10}) = 17$.
\end{theorem}
\begin{proof}
The lower bound $gr_2^3(B_{10}) \geq 17$ follows from Theorem~\ref{thm:General-Lower-Bound}. To prove the reverse inequality, consider a Gallai 3-coloring of $K_{17}$, with base graph $\mathcal{B}$, chosen to have minimal order. Since
\[
\frac{2 \cdot 17}{17 - 10 + 1} < 5 \quad \text{and} \quad \frac{5 \cdot 10}{3} \leq 17,
\]	
then by Theorem~\ref{Gallaistruct} and Lemmas~\ref{not3}, \ref{RemoveLargeBaseGraph}, and \ref{lem:no-four}, it suffices to consider the case where $|V(\mathcal{B})|=2$. Denote the vertex sets for the blocks by $X_1$ and $X_2$, with $k_i := |X_i|$ and the indexing such that $k_1 \geq k_2$. Without loss of generality, suppose the edges joining $X_1$ and $X_2$ are red. If $k_1 \geq 10$ and $k_2 \geq 2$, then we get a $B_{10}$ whose edges use at most two colors by choosing a spine in $X_2$ and any $10$ vertices in $X_1$ as pages. Thus, it suffices to consider the following cases. 

\underline{Case 1:} Assume that $(k_1,k_2) = (9,8)$. Since $k_2 \geq r(B_1,B_2)=7$ \cite{CH3}, the subgraph induced by $X_2$ must contain either a red/blue $B_1$, or a green $B_2$. Adding the nine vertices in $X_1$ as pages, in the former case we get a red/blue $B_{10}$, whereas in the latter case we get a red/green $B_{11}$ (which contains a $B_{10}$). 

\underline{Case 2:} Assume that $(k_1,k_2) = (16,1)$. Denote the single vertex in $X_2$ by $x$. Using Theorem~\ref{Gallaistruct}, let $\mathcal{B}'$ be the base graph for the subgraph induced by $X_1$, chosen to have minimal order. Since 
\[
\frac{2 \cdot 16}{16 - 10 + 1} < 5,
\]
then by Lemmas~\ref{not3} and \ref{RemoveLargeBaseGraph}, it suffices to consider the cases where $|V(\mathcal{B}')| \in \{2,4\}$. 

Let us first assume that $|V(\mathcal{B}')|=4$. If $\mathcal{B}'$ uses the two colors other than red, then since $17 \geq \frac{5 \cdot 10 - 1}{3}$ and $x$ is a vertex incident with edges in only red, Lemma~\ref{mainlem2} implies the existence of a $B_{10}$ whose edges use at most two colors. Hence, we may assume that one of the colors used in $\mathcal{B}'$ is red. Denote the vertex sets for the blocks in $\mathcal{B}'$ by $Y_1$, $Y_2$, $Y_3$, and $Y_4$, with $\ell_i := |Y_i|$ and the indexing such that $\ell_1 \geq \ell_2 \geq \ell_3 \geq \ell_4$. Since $\ell_4 \leq \lfloor \frac{16}{4} \rfloor = 4$, it follows that $\ell_1 + \ell_2 + \ell_3 \geq 16 - 4 = 12$. Let $y$ be any vertex in $Y_4$. Taking the red edge $xy$ as a spine and adding $Y_1 \cup Y_2 \cup Y_3$ as pages, we get a book with at least $12$ pages, and whose edges use at most two colors.

Now suppose that $|V(\mathcal{B}')|=2$. 
Without loss of generality, assume that the edge in $\mathcal{B}'$ is red or blue.  Denote the vertex sets for the blocks in $\mathcal{B}'$ by $Y_1$ and $Y_2$, with $\ell_i := |Y_i|$ and the indexing such that $\ell_1 \geq \ell_2$. If $\ell_1 \geq 10$ and $\ell_2 \geq 2$, then a $B_{10}$ whose edges use at most two colors can be formed using an edge in $Y_2$ as a spine with any $10$ vertices in $Y_1$ as pages. Thus, it suffices to consider the following cases. 

\underline{Subcase 2.1:} Assume that $\ell_1 \leq 9$ and $\ell_2 \geq 2$. Since $\ell_1 \leq 9$, we have that $\ell_2 \geq 16 - 9 = 7$. As $\ell_2 \geq r(B_1,B_2)=7$ \cite{CH3}, we know that the subgraph induced by $Y_2$ contains either a red/blue $B_1$, or a green $B_2$. Note that $\ell_1 \geq \lceil \frac{16}{2} \rceil = 8$. In the former case, using the vertex $x$ and all vertices in $Y_1$ as pages yields a red/blue book with at least $10$ pages. In the latter case, using the vertices in $Y_1$ as pages yields a blue/green book with at least $10$ pages. 

\underline{Subcase 2.2:} Assume that $(\ell_1,\ell_2) = (15,1)$. Denote the single vertex in $Y_2$ by $y$. Then using the red edge $xy$ as a spine, we get a red/blue $B_{15}$ (which contains a $B_{10}$) by adding the vertices in $Y_1$ as pages.

In all cases, a $B_{10}$ whose edges use at most two colors exists, from which it follows that $gr_2^3(B_{10}) \le 17$.
\end{proof}

\begin{theorem}
	$gr_2^3(B_{11}) = 19$.
\end{theorem}
\begin{proof}
The lower bound $gr_2^3(B_{11}) \geq 19$ follows from Theorem \ref{critlower} by taking $\mathcal{G}$ to be the blue/green coloring of $K_9$ given in Figure \ref{CH1.2}.
To prove the reverse inequality, consider a Gallai 3-coloring of $K_{19}$, with base graph $\mathcal{B}$, chosen to have minimal order. Since
\[
\frac{2 \cdot 19}{19 - 11 + 1} < 5 \quad \text{and} \quad \frac{5 \cdot 11}{3} \leq 19,
\]	
then by Theorem~\ref{Gallaistruct} and Lemmas~\ref{not3}, \ref{RemoveLargeBaseGraph}, and \ref{lem:no-four}, it suffices to consider the case where $|V(\mathcal{B})|=2$. Denote the vertex sets for the blocks by $X_1$ and $X_2$, with $k_i := |X_i|$ and the indexing such that $k_1 \geq k_2$. Without loss of generality, suppose the edge joining $X_1$ and $X_2$ is red. If $k_1 \geq 11$ and $k_2 \geq 2$, then we get a $B_{11}$ whose edges use at most two colors by choosing a spine in $X_2$ and any $11$ vertices in $X_1$ as pages. Thus, it suffices to consider the following cases. 

\underline{Case 1:} Assume that $(k_1,k_2) = (10,9)$. Since $k_1 = r^2(B_2)$ \cite{CH2}, the subgraph induced by $X_1$ must contain either a red/blue $B_2$, or a green $B_2$. Adding the nine vertices in $X_2$ as pages, in the former case we get a red/blue $B_{11}$, whereas in the latter case we get a red/green $B_{11}$. 

\underline{Case 2:} Assume that $(k_1,k_2) = (18,1)$. Denote the single vertex in $X_2$ by $x$. Using Theorem~\ref{Gallaistruct}, let $\mathcal{B}'$ be the base graph for the induced graph on $X_1$, chosen to have minimal order. Since 
\[
\frac{2 \cdot 18}{18 - 11 + 1} < 5,
\]
then by Lemmas~\ref{not3} and \ref{RemoveLargeBaseGraph}, it suffices to consider the cases where $|V(\mathcal{B}')| \in \{2,4\}$. 

Let us first assume that $|V(\mathcal{B}')|=4$. If $\mathcal{B}'$ uses the two colors other than red, then since $19 \geq \frac{5 \cdot 11 - 1}{3}$ and $x$ is a vertex incident with edges in only red, Lemma~\ref{mainlem2} implies the existence of a $B_{11}$ whose edges use at most two colors. Hence, we may assume that one of the colors used in $\mathcal{B}'$ is red. Denote the vertex sets for the blocks in $\mathcal{B}'$ by $Y_1$, $Y_2$, $Y_3$, and $Y_4$, with $\ell_i := |Y_i|$ and the indexing such that $\ell_1 \geq \ell_2 \geq \ell_3 \geq \ell_4$. Since $\ell_4 \leq \lfloor \frac{18}{4} \rfloor = 4$, it follows that $\ell_1 + \ell_2 + \ell_3 \geq 18 - 4 = 14$. Let $y$ be any vertex in $Y_4$. Taking the red edge $xy$ as a spine and adding $Y_1 \cup Y_2 \cup Y_3$ as pages, we get a book with at least $14$ pages, and whose edges use at most two colors.

Now suppose that $|V(\mathcal{B}')|=2$. 
Without loss of generality, assume that the edge in $\mathcal{B}'$ is red or blue.  Denote the vertex sets for the blocks in $\mathcal{B}'$ by $Y_1$ and $Y_2$, with $\ell_i := |Y_i|$ and the indexing such that $\ell_1 \geq \ell_2$. If $\ell_1 \geq 11$ and $\ell_2 \geq 2$ then a $B_{11}$ whose edges use at most two colors can be formed using an edge in $Y_2$ as a spine with any $11$ vertices in $Y_1$ as pages. Thus, it suffices to consider the following subcases. 

\underline{Subcase 2.1:} Assume that $\ell_1 \leq 10$ and $\ell_2 \geq 2$. Since $\ell_1 \leq 10$, we have that $\ell_2 \geq 18 - 10 = 8$. As $\ell_2 \geq r(B_1,B_2)$ \cite{CH3}, we know that the subgraph induced by $Y_2$ contains either a red/blue $B_1$, or a green $B_2$. Note that $\ell_1 \geq \lceil \frac{18}{2} \rceil = 9$. In the former case, using the vertex $x$ and all vertices in $Y_1$ as pages yields a red/blue book with at least $11$ pages. In the latter case, using the vertices in $Y_1$ as pages yields a blue/green book with at least $11$ pages. 

\underline{Subcase 2.2:} Assume that $(\ell_1,\ell_2) = (17,1)$. Denote the single vertex in $Y_2$ by $y$. Then using the red edge $xy$ as a spine, we get a red/blue $B_{17}$ (which contains a $B_{11}$) by adding the vertices in $Y_1$ as pages.

In all cases, a $B_{11}$ whose edges use at most two colors exists, from which it follows that $gr_2^3(B_{11}) \le 19$.
\end{proof}

In order to obtain the lower bound for $gr_2^3(B_{12})$, we will need to consider the critical coloring for $r^2(B_3)=14$ described in \cite{RS}.  Start by considering the Paley graph of order $13$, which is the $6$-regular graph shown in Figure \ref{RS1.1}.
\begin{figure}[h!]
\centering
\begin{subfigure}{.4\textwidth}
    \centering
    \includegraphics[width=1\linewidth]{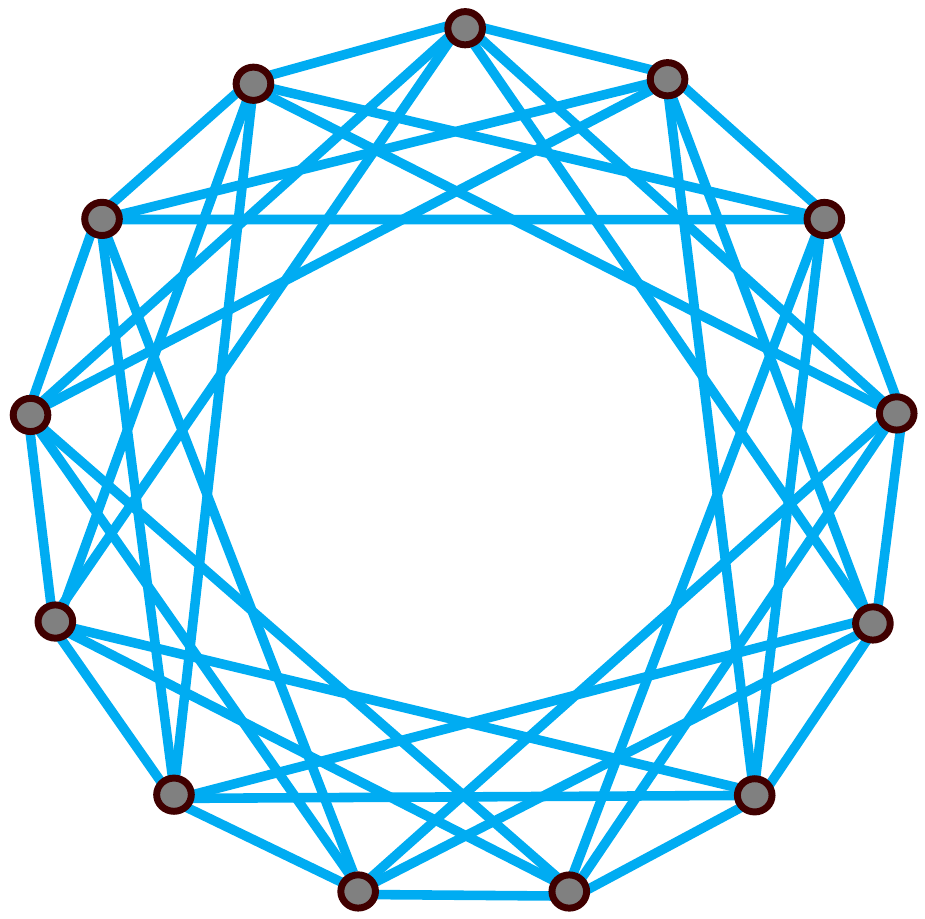}  
    \caption{}
    \label{RS1.1}
\end{subfigure}\qquad \qquad
\begin{subfigure}{.4\textwidth}
    \centering
    \includegraphics[width=1\linewidth]{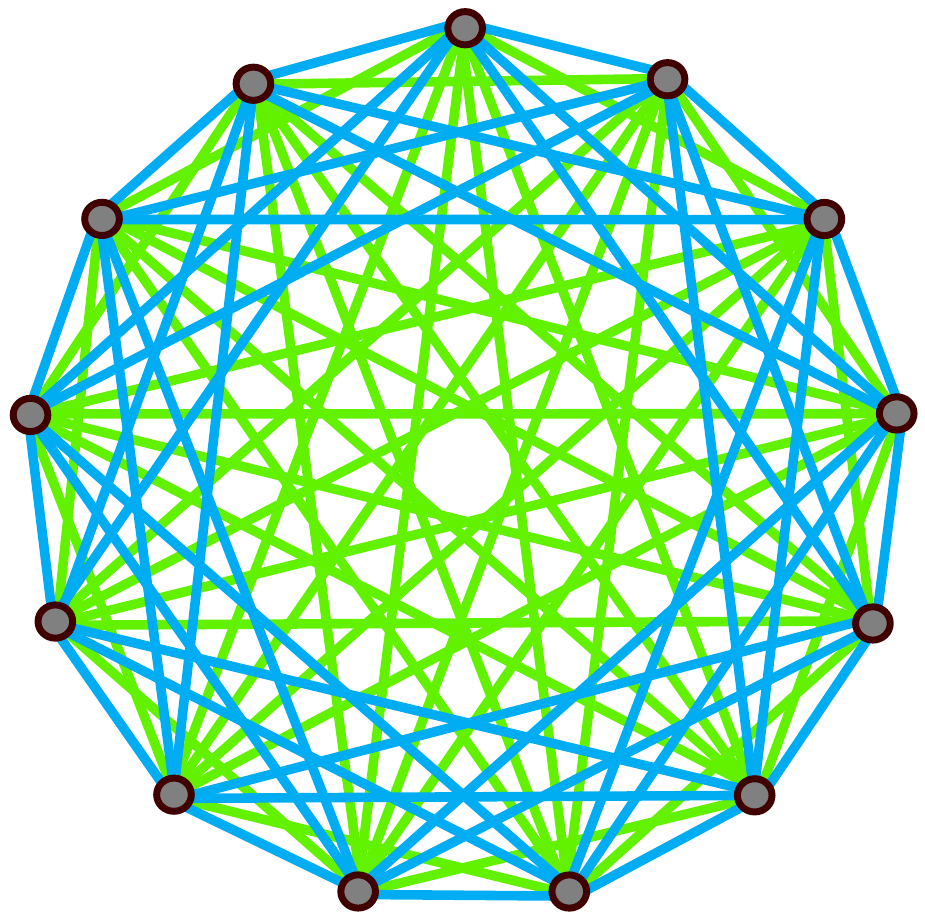}  
    \caption{}
    \label{RS1.2}
\end{subfigure}
\caption{The critical coloring for $r^2(B_3)=14$ given by Rousseau and Sheehan in \cite{RS}.}\label{RSlowerbound}
\end{figure}
It is self-complementary, and Figure \ref{RS1.2} shows this graph and its complement in blue and green, which is a critical coloring for $r^2(B_3)$.  

\begin{theorem}
	$gr_2^3(B_{12}) = 20$.
\end{theorem}
\begin{proof}
In order to obtain the lower bound, begin by letting $\mathcal{G}_1$ be the blue/green coloring of $K_9$ given in Figure \ref{CH1.2} and $\mathcal{G}_2$ be the blue/green coloring of $K_{10}$ formed by taking the graph in Figure \ref{RS1.2}, and deleting three of its vertices.  Next, take a red $K_2$ and replace one of its vertices with $G_1$ and the other vertex with $G_2$, forming a Gallai $3$-coloring of $K_{19}$.  The largest red/blue book with a red spine has $10$ pages, the largest red/blue book with a blue spine has $11$ pages, and the same is true for the largest red/green books.  The largest blue/green book has eight pages.  It follows that $gr_2^3(B_{12}) \geq 20$.

To prove the reverse inequality, consider a Gallai 3-coloring of $K_{20}$, with base graph $\mathcal{B}$, chosen to have minimal order. Since
\[
\frac{2 \cdot 20}{20 - 12 + 1} < 5 \quad \text{and} \quad \frac{5 \cdot 12}{3} \leq 20,
\]	
then by Theorem~\ref{Gallaistruct} and Lemmas~\ref{not3}, \ref{RemoveLargeBaseGraph}, and \ref{lem:no-four}, it suffices to consider the case where $|V(\mathcal{B})|=2$. Denote the vertex sets for the blocks by $X_1$ and $X_2$, with $k_i := |X_i|$ and the indexing such that $k_1 \geq k_2$. Without loss of generality, suppose the edge joining $X_1$ and $X_2$ is red. If $k_1 \geq 12$ and $k_2 \geq 2$, then we get a $B_{12}$ whose edges use at most two colors by choosing a spine in $X_2$ and any 12 vertices in $X_1$ as pages. Thus, it suffices to consider the following cases. 

\underline{Case 1:} Assume that $(k_1,k_2) = (10,10)$. Since $k_2=r^2(B_2) $ \cite{CH2}, the subgraph induced by $X_2$ must contain either a red/blue $B_2$, or a green $B_2$.  Adding the $10$ vertices in $X_1$ as pages, in the former case we get a red/blue $B_{12}$, whereas in the latter case we get a red/green $B_{12}$. 

\underline{Case 2:} Assume that $(k_1,k_2) = (11,9)$. Since $k_2\ge r(B_1,B_2) = 7$ \cite{CH3}, the subgraph induced by $X_2$ must contain either a red/blue $B_1$, or a green $B_2$. Adding the $11$ vertices in $X_1$ as pages, in the former case we get a red/blue $B_{12}$, whereas in the latter case we get a red/green $B_{13}$ (which contains a $B_{12}$). 

\underline{Case 3:} Assume that $(k_1,k_2) = (19,1)$. Denote the single vertex in $X_2$ by $x$. Using Theorem~\ref{Gallaistruct}, let $\mathcal{B}'$ be the base graph for the induced graph on $X_1$, chosen to have minimal order. Since 
\[
\frac{2 \cdot 19}{19 - 12 + 1} < 5,
\]
then by Lemmas~\ref{not3} and \ref{RemoveLargeBaseGraph}, it suffices to consider the cases where $|V(\mathcal{B}')| \in \{2,4\}$. 

Let us first assume that $|V(\mathcal{B}')|=4$. If $\mathcal{B}'$ uses the two colors other than red, then since $20 \geq \frac{5 \cdot 12 - 1}{3}$ and $x$ is a vertex incident with edges in only red, Lemma~\ref{mainlem2} implies the existence of a $B_{12}$ whose edges use at most two colors. Hence, we may assume that one of the colors used in $\mathcal{B}'$ is red. Denote the vertex sets for the blocks in $\mathcal{B}'$ by $Y_1$,$Y_2$, $Y_3$, and$Y_4$, with $\ell_i := |Y_i|$ and the indexing such that $\ell_1 \geq \ell_2 \ge \ell_3 \geq \ell_4$. Since $\ell_4 \leq \lfloor \frac{19}{4} \rfloor = 4$, it follows that $\ell_1 + \ell_2 + \ell_3 \geq 19 - 4 = 15$. Let $y$ be any vertex in $Y_4$. Taking the red edge $xy$ as a spine and adding $Y_1 \cup Y_2 \cup Y_3$ as pages, we get a book with at least $15$ pages, and whose edges use at most two colors.

Now suppose that $|V(\mathcal{B}')|=2$. 
Without loss of generality, assume that the edge in $\mathcal{B}'$ is red or blue.  Denote the vertex sets for the blocks in $\mathcal{B}'$ by $Y_1$ and $Y_2$, with $\ell_i := |Y_i|$ and the indexing such that $\ell_1 \geq \ell_2$. If $\ell_1 \geq 12$ and $\ell_2 \geq 2$ then a $B_{12}$ whose edges use at most two colors can be formed using an edge in $Y_2$ as a spine with any $12$ vertices in $Y_1$ as pages. Thus, it suffices to consider the following subcases. 

\underline{Subcase 3.1:} Assume that $\ell_1 \leq 11$ and $\ell_2 \geq 2$. Since $\ell_1 \leq 11$, we have that $\ell_2 \geq 19 - 11 = 8$. As $\ell_2 \geq r(B_1,B_2)=7$ \cite{CH3}, we know that the subgraph induced by $Y_2$ contains either a red/blue $B_1$, or a green $B_2$. Note that $\ell_1 \geq \lceil \frac{19}{2} \rceil = 10$. In the former case, using the vertex $x$ and all vertices in $Y_1$ as pages yields a red/blue book with at least $12$ pages. In the latter case, using the vertices in $Y_1$ as pages yields a blue/green book with at least $12$ pages. 

\underline{Subcase 3.2:} Assume that $(\ell_1,\ell_2) = (18,1)$. Denote the single vertex in $Y_2$ by $y$. Then using the red edge $xy$ as a spine, we get a red/blue $B_{18}$ (which contains a $B_{12}$) by adding the vertices in $Y_1$ as pages.

In all cases, a $B_{12}$ whose edges use at most two colors exists, from which it follows that $gr_2^3(B_{12}) \le 20$.
\end{proof}

\begin{theorem}
	$gr_2^3(B_{13}) = 22$. 
\end{theorem}
\begin{proof}
The lower bound $gr_2^3(B_{13}) \geq 22$ follows from Theorem~\ref{thm:General-Lower-Bound}. To prove the reverse inequality, consider a Gallai 3-coloring of $K_{22}$, with base graph $\mathcal{B}$, chosen to have minimal order. Since
\[
\frac{2 \cdot 22}{22 - 13 + 1} < 5 \quad \text{and} \quad \frac{5 \cdot 13}{3} \leq 22,
\]	
then by Theorem~\ref{Gallaistruct} and Lemmas~\ref{not3}, \ref{RemoveLargeBaseGraph}, and \ref{lem:no-four}, it suffices to consider the case where $|V(\mathcal{B})|=2$. Denote the vertex sets for the blocks by $X_1$ and $X_2$, with $k_i := |X_i|$ and the indexing such that $k_1 \geq k_2$. Without loss of generality, suppose the edge joining $X_1$ and $X_2$ is red. If $k_1 \geq 13$ and $k_2 \geq 2$ then we get a $B_{13}$ whose edges use at most two colors by choosing a spine in $X_2$ and any $13$ vertices in $X_1$ as pages. Thus, it suffices to consider the following cases. 

\underline{Case 1:} Assume that $k_1 \leq 12$ and $k_2 \geq 2$. Since $k_1 \leq 12$, we have that $k_2 \geq 22 - 12 = 10$. As $k_2 \geq r^2(B_2)$ \cite{CH2}, we know that the subgraph induced by $X_2$ contains either a red/blue $B_2$, or a green $B_2$. Since $k_1 \geq \lceil \frac{22}{2} \rceil = 11$, then adding the vertices in $X_1$ as pages yields a book with at least $13$ pages, and whose edges use at most two colors (either red/blue or red/green). 

\underline{Case 2:} Assume that $(k_1,k_2) = (21,1)$. Denote the single vertex in $X_2$ by $x$. Using Theorem~\ref{Gallaistruct}, let $\mathcal{B}'$ be the base graph for the induced graph on $X_1$, chosen to have minimal order. Since 
\[
\frac{2 \cdot 21}{21 - 13 + 1} < 5,
\]
then by Lemmas~\ref{not3} and \ref{RemoveLargeBaseGraph}, it suffices to consider the cases where $|V(\mathcal{B}')| \in \{2,4\}$. 

Let us first assume that $|V(\mathcal{B}')|=4$. If $\mathcal{B}'$ uses the two colors other than red, then since $22 \geq \frac{5 \cdot 13 - 1}{3}$ and $x$ is a vertex incident with edges in only red, Lemma~\ref{mainlem2} implies the existence of a $B_{13}$ whose edges use at most two colors. Hence, we may assume that one of the colors used in $\mathcal{B}'$ is red. Denote the vertex sets for the blocks in $\mathcal{B}'$ by $Y_1$, $Y_2$, $Y_3$, and $Y_4$, with $\ell_i := |Y_i|$ and the indexing such that $\ell_1 \geq \ell_2 \geq \ell_3 \geq \ell_4$. Since $\ell_4 \leq \lfloor \frac{21}{4} \rfloor = 5$, it follows that $\ell_1 + \ell_2 + \ell_3 \geq 21 - 5 = 16$. Let $y$ be any vertex in $Y_4$. Taking the red edge $xy$ as a spine and adding $Y_1 \cup Y_2 \cup Y_3$ as pages, we get a book with at least $16$ pages, and whose edges use at most two colors.

Now suppose that $|V(\mathcal{B}')|=2$. 
Without loss of generality, assume that the edge in $\mathcal{B}'$ is red or blue.  Denote the vertex sets for the blocks in $\mathcal{B}'$ by $Y_1$ and $Y_2$, with $\ell_i := |Y_i|$ and the indexing such that $\ell_1 \geq \ell_2$. If $\ell_1 \geq 13$ and $\ell_2 \geq 2$, then a $B_{13}$ whose edges use at most two colors can be formed using an edge in $Y_2$ as a spine with any $13$ vertices in $Y_1$ as pages. Thus, it suffices to consider the following subcases. 

\underline{Subcase 2.1:} Assume that $\ell_1 \leq 12$ and $\ell_2 \geq 2$. Since $\ell_1 \leq 12$, we have that $\ell_2 \geq 21 - 12 = 9$. As $\ell_2 \geq r(B_1,B_2)=7$ \cite{CH3}, we know that the subgraph induced by $Y_2$ contains either a red/blue $B_1$, or a green $B_2$. Note that $\ell_1 \geq \lceil \frac{21}{2} \rceil = 11$. In the former case, using the vertex $x$ and all vertices in $Y_1$ as pages yields a red/blue book with at least $13$ pages. In the latter case, using the vertices in $Y_1$ as pages yields a blue/green book with at least $13$ pages. 

\underline{Subcase 2.2:} Assume that $(\ell_1,\ell_2) = (20,1)$. Denote the single vertex in $Y_2$ by $y$. Then using the red edge $xy$ as a spine, we get a red/blue $B_{20}$ (which contains a $B_{13}$) by adding the vertices in $Y_1$ as pages.

In all cases, a $B_{13}$ whose edges use at most two colors exists, from which it follows that $gr_2^3(B_{13}) \le 22$.
\end{proof}

\begin{theorem}
	$gr_2^3(B_{14}) = 23$. 
\end{theorem}
\begin{proof}
The lower bound $gr_2^3(B_{14}) \geq 23$ follows from the construction shown in Figure \ref{LowerB14}. 
\begin{figure}[h!]
\centerline{
\includegraphics[width=0.55\textwidth]{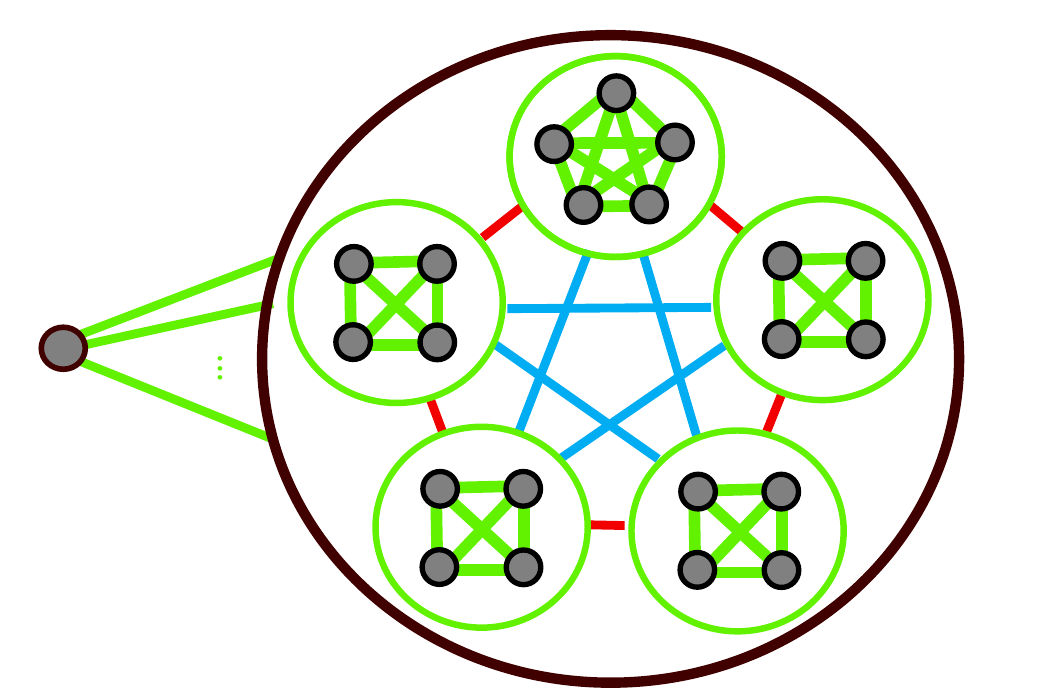}}
\caption{A Gallai $3$-coloring of $K_{22}$ that avoids a $B_{14}$ whose edges use at most two colors.}\label{LowerB14}
\end{figure}
Here, the largest red/blue book has $13$ pages, the largest red/green book has $12$ pages, and the largest blue/green book has $12$ pages. 
To prove the reverse inequality, consider a Gallai 3-coloring of $K_{23}$, with base graph $\mathcal{B}$, chosen to have minimal order. Since
\[
\frac{2 \cdot 23}{23 - 14 + 1} < 5,
\]	
then by Theorem~\ref{Gallaistruct} and Lemmas~\ref{not3}, and \ref{RemoveLargeBaseGraph}, it suffices to consider the cases where $|V(\mathcal{B})|\in \{2,4\}$.  We consider the following cases. 

\underline{Case 1:} Assume that $|V(\mathcal{B})| = 4$. Without loss of generality, suppose that the $2$-coloring of $\mathcal{B}$ uses red and blue.  Denote the vertex sets for the blocks by $X_1$, $X_2$, $X_3$, and $X_4$ with $k_i := |X_i|$ and the indexing such that $k_1 \geq k_2\geq k_3 \geq k_4$. If $k_1 + k_2 \geq 14$, then a $B_{14}$ whose edges use at most two colors is formed by taking an edge between $X_3$ and $X_4$ as a spine and choosing any $14$ vertices in $X_1 \cup X_2$ as pages. Therefore, we may assume that $k_1 + k_2 \leq 13$.  
If $k_1\ge 9$, then $k_2\le 13-k_1$ implies that $$23=k_1+k_2+k_3+k_4\le k_1+3(13-k_1)\le -2(9)+39=21,$$ which is a contradiction.  Thus, we find that $\ceil{\frac{23}{4}}=6\le k_1\le 8$, from which it follows that $15 \ge k_2+k_3+k_4\le 17$.  In particular, note that if any $X_i$ contains a red or blue edge, then using that edge as a spine, a red/blue book with at least $15$ pages can be formed by using the vertices in the other three blocks as pages.  So, assume that every block is a complete green graph.  Since $k_4\le \floor{\frac{23}{4}}=5$, it follows that $k_1+k_2+k_3\ge 18$.  In the subgraph of $\mathcal{B}$ induced by the vertices corresponding to $X_1$, $X_2$, and $X_3$, at least two of the edges must be the same color.  If the edges joining $X_i$ to $X_j\cup X_k$ (where $\{i,j,k\}=\{1,2,3\}$) are all the same color (say, red), then a red/green $B_{16}$ can be formed by picking a spine $xy$ in $X_i$ and pages in $(X_i\setminus \{x,y\})\cup X_j\cup X_k$.

\underline{Case 2:} Assume that $|V(\mathcal{B})| = 2$.  Without loss of generality, suppose that the edge in $\mathcal{B}$ is red.  Denote the vertex sets for the blocks by $X_1$ and $X_2$, with $k_i := |X_i|$ and the indexing such that $k_1 \geq k_2$.  If $k_2 \geq 2$ and $k_1 \geq 14$, then we get a $B_{14}$ whose edges use at most two colors by choosing a spine in $X_2$ and any $14$ vertices in $X_1$ as pages. Thus, it suffices to consider the following subcases. 

\underline{Subcase 2.1:} Assume that $12\le k_1\le 13$ and $10\le k_2\le 11$.  As $k_2\ge r^2(B_2)=10$ \cite{CH2}, it follows that $X_2$ contains a red/blue $B_2$ or a green $B_2$.  In the former case, adding in the vertices in $X_1$ as pages results in a red/blue book with at least $14$ pages.  In the latter case, adding in the vertices in $X_1$ as pages results in a red/green book with at least $14$ pages.

\underline{Subcase 2.2:} Assume that $(k_1,k_2) = (22,1)$.  Denote the single vertex in $X_2$ by $x$. Using Theorem~\ref{Gallaistruct}, let $\mathcal{B}'$ be the base graph for the subgraph induced by $X_1$, chosen to have minimal order. Since 
\[
\frac{2 \cdot 23}{23 - 14 + 1} < 5 \quad \mbox{and} \quad \frac{5\cdot 14-1}{3}\le 23,
\]
then by Lemmas~\ref{not3}, \ref{RemoveLargeBaseGraph}, and \ref{mainlem2}, it suffices to consider the case where $|V(\mathcal{B}')| =2$.   Without loss of generality, assume that the edge in $\mathcal{B}'$ is red or blue.   Denote the vertex sets for the blocks in $\mathcal{B}'$ by $Y_1$ and $Y_2$, with $\ell_i := |Y_i|$ and the indexing such that $\ell_1 \geq \ell_2$. If $\ell_1 \geq 14$ and $\ell_2 \geq 2$, then a $B_{14}$ whose edges use at most two colors can be formed using an edge in $Y_2$ as a spine with any $14$ vertices in $Y_1$ as pages. Thus, it suffices to consider the following subsubcases. 

\underline{Subsubcase 2.2.1:} Assume that $(\ell_1,\ell_2)=(11,11)$.  As $\ell_2=r(B_2,B_3)$ (see \cite{Clan} and \cite{SXBP}), the subgraph induced by $Y_2$ contains a red/blue $B_2$ or a green $B_3$.  In the former case, a red/blue $B_{14}$ can be formed by adding in the vertices in $\{x\}\cup Y_1$ as pages.  In the latter case, a $B_{14}$ whose edges use at most two colors (one of which is green) can be formed by adding in the vertices in $Y_1$ as pages.

\underline{Subsubcase 2.2.2:} Assume that $(\ell_1,\ell_2)=(12,10)$.  As $k_2=r^2(B_2)$ \cite{CH2}, the subgraph induced by $Y_2$ contains a red/blue $B_2$ or a green $B_2$.  In the former case, a red/blue $B_{14}$ can be formed by adding in the vertices in $Y_1$ as pages.  In the latter case, a $B_{14}$ whose edges use at most two colors (one of which is green) can be formed by adding in the vertices in $Y_1$ as pages.

\underline{Subsubcase 2.2.3:} Assume that $(\ell_1,\ell_2)=(13,9)$. As $k_2\ge r^2(B_1)$ \cite{GG}, the subgraph induced by $Y_2$ contains a red/blue $B_1$ or a green $B_1$.  In the former case, a red/blue $B_{14}$ can be formed by adding in the vertices in $Y_1$ as pages.  In the latter case, a $B_{14}$ whose edges use at most two colors (one of which is green) can be formed by adding in the vertices in $Y_1$ as pages.

\underline{Subsubcase 2.2.4:} Assume that $(\ell_1,\ell_2)=(21,1)$.  Denote the single vertex in $Y_2$ by $y$.  Then using the red edge $xy$ as a spine, we get a red/blue $B_{21}$ (which contains a $B_{14}$) by adding the vertices in $Y_1$ as pages.

In all cases, a $B_{14}$ whose edges use at most two colors exists, from which it follows that $gr_2^3(B_{14})\le 23$.
\end{proof}

\begin{theorem}
	$gr_2^3(B_{15})=25$. 
\end{theorem}
\begin{proof}
In order to obtain the lower bound, begin by letting $\mathcal{G}$ be the blue/green coloring of $K_{12}$ 
formed by taking the graph in Figure \ref{RS1.2} and deleting a single vertex.  
Next, take a red $K_2$ and replace both of its vertices with copies of $\mathcal{G}$, forming a Gallai $3$-coloring of $K_{24}$.  The largest red/blue book with a red spine has $12$ pages, the largest red/blue book with a blue spine has $14$ pages, and the same is true for the largest red/green books.  The largest blue/green book has $10$ pages.  It follows that $gr_2^3(B_{15}) \geq 25$.

To prove the reverse inequality, consider a Gallai 3-coloring of $K_{25}$, with base graph $\mathcal{B}$, chosen to have minimal order. Since
\[
\frac{2 \cdot 25}{25 - 15 + 1} < 5 \quad \text{and} \quad \frac{5 \cdot 15}{3} \leq 25,
\]	
then by Theorem~\ref{Gallaistruct} and Lemmas~\ref{not3}, \ref{RemoveLargeBaseGraph}, and \ref{lem:no-four}, it suffices to consider the case where $|V(\mathcal{B})|=2$. Denote the vertex sets for the blocks by $X_1$ and $X_2$, with $k_i := |X_i|$ and the indexing such that $k_1 \geq k_2$. Without loss of generality, suppose the edge joining $X_1$ and $X_2$ is red. If $k_1 \geq 15$ and $k_2 \geq 2$ then we get a $B_{15}$ whose edges use at most two colors by choosing a spine in $X_2$ and any $14$ vertices in $X_1$ as pages. Thus, it suffices to consider the following cases. 

\underline{Case 1:} Assume that $k_1 \leq 14$ and $k_2 \geq 2$. Since $k_1 \leq 14$, we have that $k_2 \geq 25 - 14 = 11$. As $k_2 \geq r^2(B_2)=10$ \cite{CH2}, we know that the subgraph induced by $X_2$ contains either a red/blue $B_2$, or a green $B_2$. Since $k_1 \geq \lceil \frac{25}{2} \rceil = 13$, then adding the vertices in $X_1$ as pages yields a book with at least $15$ pages, and whose edges use at most two colors (either red/blue or red/green). 

\underline{Case 2:} Assume that $(k_1,k_2) = (24,1)$. Denote the single vertex in $X_2$ by $x$. Using Theorem~\ref{Gallaistruct}, let $\mathcal{B}'$ be the base graph for the induced graph on $X_1$, chosen to have minimal order. Since 
\[
\frac{2 \cdot 23}{23 - 14 + 1} < 5,
\]
then by Lemmas~\ref{not3} and \ref{RemoveLargeBaseGraph}, it suffices to consider the cases where $|V(\mathcal{B}')| \in \{2,4\}$. 

Let us first assume that $|V(\mathcal{B}')|=4$. If $\mathcal{B}'$ uses the two colors other than red, then since $25 \geq \frac{5 \cdot 14 - 1}{3}$ and $x$ is a vertex incident with edges in only red, Lemma~\ref{mainlem2} implies the existence of a $B_{15}$ whose edges use at most two colors. Hence, we may assume that one of the colors used in $\mathcal{B}'$ is red. Denote the vertex sets for the blocks in $\mathcal{B}'$ by $Y_1$, $Y_2$, $Y_3$, and $Y_4$, with $\ell_i := |Y_i|$ and the indexing such that $\ell_1 \geq \ell_2 \geq \ell_3 \geq \ell_4$. Since $\ell_4 \leq \frac{24}{4} = 6$, it follows that $\ell_1 + \ell_2 + \ell_3 \geq 24 - 6 = 18$. Let $y$ be any vertex in $Y_4$. Taking the red edge $xy$ as a spine and adding $Y_1 \cup Y_2 \cup Y_3$ as pages, we get a book with at least $18$ pages, and whose edges use at most two colors.

Now suppose that $|V(\mathcal{B}')|=2$. 
Without loss of generality, assume that the edge in $\mathcal{B}'$ is red or blue.  Denote the vertex sets for the blocks in $\mathcal{B}'$ by $Y_1$ and $Y_2$, with $\ell_i := |Y_i|$ and the indexing such that $\ell_1 \geq \ell_2$. If $\ell_1 \geq 15$ and $\ell_2 \geq 2$, then a $B_{15}$ whose edges use at most two colors can be formed using an edge in $Y_2$ as a spine with any $14$ vertices in $Y_1$ as pages. Thus, it suffices to consider the following subcases. 

\underline{Subcase 2.1:} Assume that $(\ell_1,\ell_2) = (12,12)$.  As $\ell_2 \geq r(B_2,B_3)=11$ (see \cite{Clan} and \cite{SXBP}), we know that the subgraph induced by $Y_2$ contains either a red/blue $B_2$, or a green $B_3$. In the former case, adding the vertex $x$ along with the vertices in $Y_1$ yields a red/blue book with at least $15$ pages, whereas in the latter case, adding the vertices in $Y_1$ yields a blue/green book with at least $15$ pages.

\underline{Subcase 2.2:} Assume that $13 \leq \ell_1 \leq 14$ and $\ell_2 \geq 2$. Since $\ell_1 \leq 14$, we have that $\ell_2 \geq 24 - 14 = 10$.  As $\ell_2 \geq r^2(B_2)$ \cite{CH2}, we know that the subgraph induced by $Y_2$ contains either a red/blue $B_2$, or a green $B_2$. Hence, adding the vertices in $Y_1$ as pages yields a book with at least $15$ pages, and whose edges use at most two colors. 

\underline{Subcase 2.3:} Assume that $(\ell_1,\ell_2) = (23,1)$. Denote the single vertex in $Y_2$ by $y$. Then using the red edge $xy$ as a spine, we get a red/blue $B_{23}$ (which contains a $B_{15}$) by adding the vertices in $Y_1$ as pages.

In all cases, a $B_{15}$ whose edges use at most two colors exists, from which it follows that $gr_2^3(B_{15}) \le 25$.
\end{proof}

\section{Conclusion}\label{conclusion}
We are optimistic that the methods presented in Section~\ref{main} could be continued for larger values of n, repeating the same techniques, with an increasing number of cases. So far, we have been unable to generalize our arguments, which often relied on the known values of $r(B_{n-1},B_n)$ and $r^2(B_n)$. One could certainly continue this pursuit, exhausting the known values of $2$-color Ramsey numbers for books. A different interesting direction that could be considered would be to determine the values of $gr^t_2(B_n)$ when $t>3$ and $n\ge 5$, extending the work initiated in Section \ref{B3andB4}.

In \cite{JM}, Jakhar and Moun observed that $$gr^t_s(K_{2,n})\le gr^t_s(B_n),$$ since $K_{2,n}$ is a subgraph of $B_n$.  When $t=3$, $s=2$, and $3\le n\le 4$, equality holds.  Since the existence of a $K_{2,n}$-subgraph whose edges use at most $s-1$ colors implies the existence of a $B_n$-subgraph whose edges use at most $s$ colors, it also follows that $$gr^t_s(B_n)\le gr^t_{s-1}(K_{2,n}).$$  One way to extend the work in this paper would be to determine the values of $gr^t_s(K_{2,n})$, for $n\ge 3$.

Finally, suppose that a Gallai $t$-coloring of $K_p$ contains a $B_n$-subgraph whose edges use at most $s$ colors, where $p>n+2$.   Assume that the $B_n$-subgraph has spine $xy$ and pages $z_1, z_2 \dots, z_n$.  Let $z$ be a vertex other than $x,y,z_1,z_2,\dots , z_n$ (which must exist since $p>n+2$).  Since the edges $xz$ and $yz$ introduce at most one new color to the $B_n$-subgraph in order to avoid a rainbow triangle, it follows that $$gr^t_{s+1}(B_{n+1})\le gr^t_s(B_n).$$  This inequality provides a good starting point for considering $gr^t_s(B_n)$ when $s>2$.

\

\noindent {\bf Acknowledgements.}  The first-listed author was funded by a Scholarly Development Assignment Program award from Western Carolina University.

\bibliographystyle{amsplain}



\end{document}